\documentclass[english]{amsart}
\usepackage{amssymb,amsmath}
\usepackage{caption}
\usepackage{float}
\usepackage{graphicx,subfigure,epsfig}
\usepackage{epstopdf}
\usepackage{mathtools}
\usepackage{multicol}
\usepackage{fancybox}
\usepackage[usenames, dvipsnames]{color}
\newtheorem{theorem}{Theorem}[section]
\newtheorem{definition}{Definition}[section]
\newtheorem{lemma}{Lemma}[section]
\newtheorem{remark}{Remark}[section]

\newtheorem{proposition}{Proposition}[section]

\DeclarePairedDelimiter\abs{\lvert}{\rvert}

\begin{document}

\markboth{K.~Kaur, M.~Prabhakar}{Virtual knots and links with unknotting index $(n,m)$}

\title{Virtual knots and links with unknotting index $(n,m)$}

\author{K. Kaur} 
\address{Department of Computational Statistics And Data Analytics, Guru Nanak Dev University, Punjab 143005, India}
\email{kirandeepoffical@gmail.com}

\author{M. Prabhakar} 
\address{Department of Mathematics, Indian Institute of Technology Ropar, Punjab 140001, India}
\email{prabhakar@iitrpr.ac.in}

\begin{abstract}
In \cite{kaur2019unknottingknot}, K. Kaur, S. Kamada {\it et al.} posed a problem of finding a virtual knot, if exists, with an  unknotting index $(n,m)$, where $(n,m)$ is a pair of non-negative integers. In this paper, we address this question by providing infinite families of virtual knots with unknotting indices $(0,m)$ and $(1,0)$, respectively. In general, we establish the existence of infinitely many distinct virtual knot diagrams with unknotting index $(n,m)$, for any pair $(n,m)$ of positive integers. Furthermore, we positively address this question for $k(>1)$-component virtual links positively by providing infinite families of $k(>1)$-component virtual links with unknotting index $(n,m)$, for a given pair of non-negative integers $(n,m)$. 
\end{abstract}
\makeatletter{\renewcommand*{\@makefnmark}{}
\makeatother}

\maketitle

\section{Introduction}
\noindent Virtual knots were introduced by L. H. Kauffman \cite{kauffman1999virtual} as a natural generalization of classical knots,
and independently by M. Goussarov, M. Polyak and O. Viro \cite{goussarov1998finite}. Various knot invariants from the literature have been naturally extended to virtual knot invariants and, more generally, to virtual link invariants, as well. In recent years, several invariants have been introduced to distinguish two given virtual knots.  These include the index polynomial \cite{im2010index}, two variable polynomials~\cite{kaur2018polynomial}, affine index \cite{kauffman2013affine}, arc shift number and region arc shift number~\cite{gill2019arc, kaur2019arc}, multi-variable polynomial \cite{miyazawa2008}, zero polynomial \cite{ju2016zero}, and $n$-th writhe \cite{SatohTaniguchi2014} of virtual knots. 

\noindent The concept of an unknotting operation holds a significant importance in the knot theory as well as in the virtual knot theory to define an  unknotting invariant. In literature, very few unknotting invariants are known for virtual knots. Although a virtual knot extends from a knot, not every unknotting operation applicable to knots necessarily applies to virtual knots. While well
known local moves such as crossing change, delta move and sharp move are unknotting operations for classical
knots \cite{kawauchi2012survey}, but for virtual knots, these local moves are not unknotting operations \cite{murakami1985some, murakami1989certain}. One of the unknotting operations for virtual knots is known as virtualization, which is a replacement of classical crossing by virtual crossing. Taking crossing change and virtualization operation together, a new unknotting invariant for virtual knots called the unknotting index has been introduced in \cite{kaur2019unknottingknot}. This work has been further extended to virtual links by K. Kaur {\it et al.} in \cite{kaur2019unknottinglink}. The unknotting index of a virtual knot is a pair of non-negative integers, compared with respect to the dictionary order. This unknotting index is an extension of the usual unknotting number. Similar to usual unknotting number, there is no general mechanism for computing the unknotting index of a given virtual knot/link. This unknotting index is known for some well known families of virtual knots and links, see \cite{kaur2019unknottingknot, kaur2019unknottinglink}. 

\noindent In \cite{kaur2019unknottingknot}, it was proved  that for any non-negative integer $n$ there exists a virtual knot whose unknotting index is $(1, n)$. For any pair $(n,m)$ of non-negative integers, following question has been arised by K. Kaur {\it et al.} in \cite{kaur2019unknottingknot},
\[ \text{\it Find a virtual knot $K$ with~} U(K)=(n,m).\] 
 In this paper, we address this problem by constructing an infinite family of distinct virtual knot diagrams whose unknotting index is $(n,m)$, for any given pair $(n,m)$ of positive integers. Specifically, we establish the existence of infinitely many virtual knots with uknotting indices $(0,m)$ and $(1,0)$ in Theorem~\ref{thm:(0,m)} and Theorem~\ref{thm:(1,0)knot}, respectively. Furthermore, in Theorem~\ref{thm:(n,m)link} and Proposition~\ref{prop1:(n,m)link}, we address this question for virtual links positively by 
 establishing the existence of an infinite family of $k$-component virtual links for any $k>1$, with unknotting index $(n,m)$. However, proving this result for virtual knots (where $k=1$) for any pair $(n,m)$ is generally difficult. We partially resolve this question in Theorem~\ref{thm:(n,m)knots}, where we establish the existence of  an infinite family of distinct virtual knot diagrams with unknotting index $(n,m)$, for any pair $(n,m)$ of positive integers. The virtual knot diagrams in this infinite family are of minimal classical crossings diagrams.  Although it seems that a diagram $D$ constructed in  Theorem~\ref{thm:(n,m)knots} provides the unknotting index of the virtual knot presented by the diagram $D$, but proving this result is challenging due to insufficient lower bounds.\\
   
\noindent This paper is organized as follows: Section~\ref{pre} contains the preliminaries that are required to prove the main results of the paper.  Specifically, we define the unknotting index for virtual knots and revisit the concept of $n$-th writhe for virtual knots. In Section~\ref{main}, we present the main results of this paper. We conclude this paper in Section~\ref{conclusion} by discussing the work accomplished, the difficulties encountered, and potential future extensions of this research.
\section{Preliminaries}
\label{pre}
\noindent A virtual link diagram exhibits two types of crossings: classical crossings and virtual crossings, see Fig.~\ref{figcv}.

\begin{figure}[!ht] 
\centering
\subfigure[Classical crossings.]
{\includegraphics[scale=0.3]{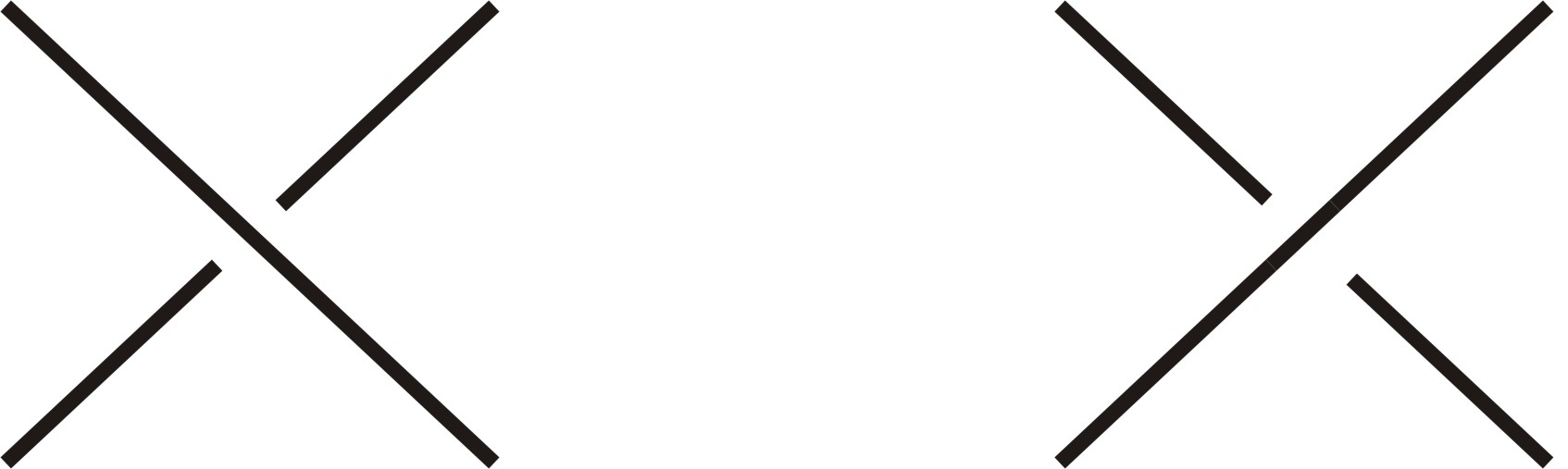} 
} \hspace{.5cm} \subfigure[Virtual crossing.]
{\hspace{1cm}
\includegraphics[scale=.3]{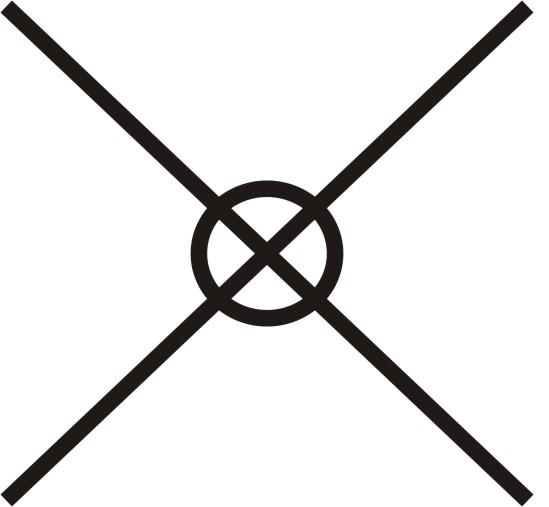}\hspace{1cm}
}
\caption{} \label{figcv}
\end{figure}

\noindent Equivalence between two virtual link diagrams is established by deforming one virtual link diagram to another through a finite sequence of classical Reidemeister moves RI, RII, RIII, as shown in Fig.~\ref{fig1a}, and virtual Reidemeister moves VRI, VRII, VRIII, SV, as shown in Fig.~\ref{fig1b}.

\begin{figure}[!ht] 
\centering
\subfigure[Classical Reidemeister moves.]
{\includegraphics[scale=0.3]{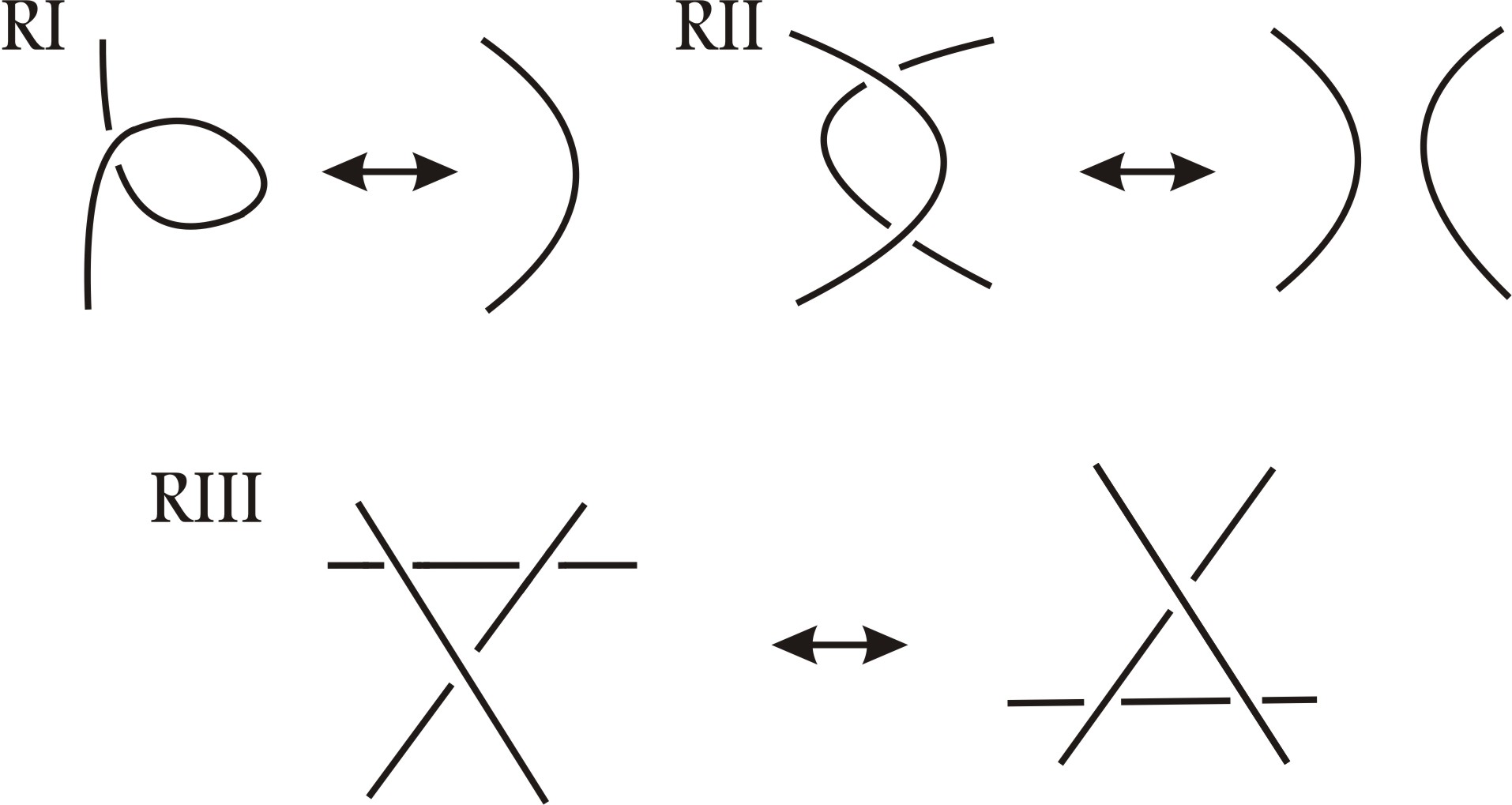} \label{fig1a}
} \hspace{.2cm} \subfigure[Virtual Reidemeister moves.]
{
\includegraphics[scale=.3]{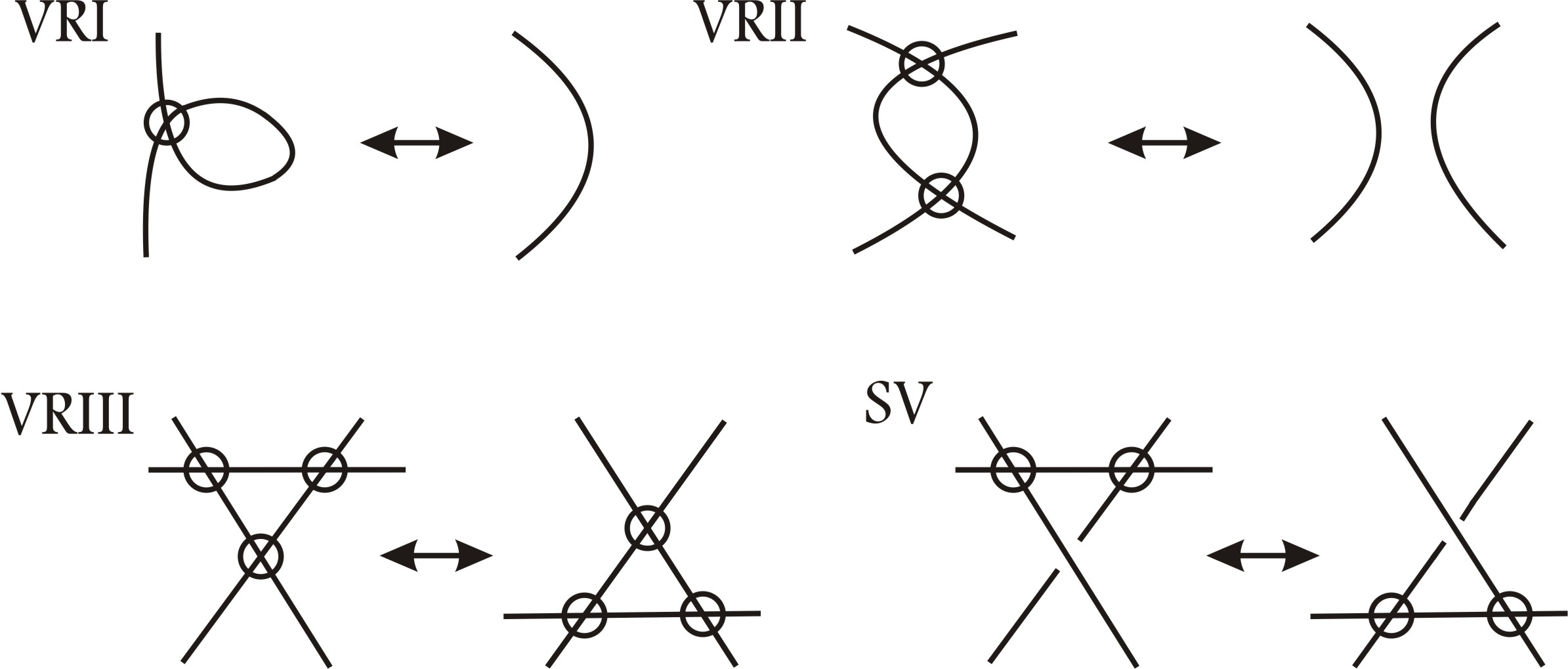} \label{fig1b}
}
\caption{Reidemeister moves.} \label{fig-rei}
\end{figure}

\noindent A diagram $D$ of a virtual knot $K$ is said to be $(n, m)$-unknottable if by virtualizing $n$ classical crossings and crossing change operation at $m$ classical crossings of $D$, the resulting diagram can be deformed into a diagram of the trivial knot. Note that the diagram $D$ is $(|\mathcal{C}(D)|, 0)$-unknottable, where $\mathcal{C}(D)$ is the set of classical crossings of $D$ and $|\mathcal{C}(D)|$ is the cardinality of the set $\mathcal{C}(D)$. \\

\begin{definition}\cite{kaur2019unknottingknot} {\rm 
The {\it unknotting index}  of a diagram $D$, denoted by $U(D)$, is the minimum among all pairs $(n,m)$ such that $D$ is $(n,m)$-unknottable. Here the minimality is taken with respect to the dictionary order.}\end{definition}
 \begin{definition}\cite{kaur2019unknottingknot} {\rm 
The {\it unknotting index} of a virtual knot $ K$, denoted by $U(K)$, is defined as
\[U(K)=\min\{U(D)| D\in [K]\},\]
where the minimality is taken with respect to the dictionary order.}\end{definition}

\noindent The unknotting index for a virtual link is defined similarly, as detailed in \cite[Definition~1.1]{kaur2019unknottinglink}. 
A virtual knot or link $L$ is trivial if and only if $U (L) = (0, 0)$. For a classical link $L$, the unknotting index is $U (L) = (0, u(L))$, where $u(L)$ is the usual unknotting number of $L$. 

\noindent In general, determining the unknotting index for a given virtual knot or link is a challenging task. Lower bounds on this unknotting index have been established  for virtual knots and links in \cite{kaur2019unknottingknot, kaur2019unknottinglink},  using $n$-th writhe invariant \cite{SatohTaniguchi2014}, and span value \cite{cheng2013polynomial, kaur2019unknottinglink}. These lower bounds have been used to compute the unknotting index for some well known families of virtual knots and links, as shown in \cite{kaur2019unknottingknot, kaur2019unknottinglink}.

\noindent Let $D$ be an oriented virtual knot diagram. An \emph{arc} is  an edge between two consecutive classical crossings along the orientation. 
\noindent Now assign an integer value to each arc in $D$ in such a way that the labeling around each crossing point of $D$ follows the rule shown in Fig.~\ref{Fig1fig3}. L.~H.~Kauffman proved \cite[Proposition~4.1]{kauffman2013affine} that such integer labeling, called a \emph{Cheng coloring}, always exists for an oriented virtual knot diagram. 
\begin{figure}[!ht]
\centering 
\unitlength=0.6mm
\begin{picture}(0,30)(0,5)
\thicklines
\qbezier(-70,10)(-70,10)(-50,30)
\qbezier(-70,30)(-70,30)(-62,22) 
\qbezier(-50,10)(-50,10)(-58,18)
\put(-65,25){\vector(-1,1){5}}
\put(-55,25){\vector(1,1){5}}
\put(-53,20){\makebox(0,0)[cc]{$c$}}
\put(-75,34){\makebox(0,0)[cc]{$b+1$}}
\put(-75,8){\makebox(0,0)[cc]{$a$}}
\put(-45,8){\makebox(0,0)[cc]{$b$}}
\put(-45,34){\makebox(0,0)[cc]{$a-1$}}
\qbezier(10,10)(10,10)(-10,30)
\qbezier(10,30)(10,30)(2,22) 
\qbezier(-10,10)(-10,10)(-2,18)
\put(-5,25){\vector(-1,1){5}}
\put(5,25){\vector(1,1){5}}
\put(7,20){\makebox(0,0)[cc]{$c$}}
\put(-15,34){\makebox(0,0)[cc]{$b+1$}}
\put(-15,8){\makebox(0,0)[cc]{$a$}}
\put(15,8){\makebox(0,0)[cc]{$b$}}
\put(15,34){\makebox(0,0)[cc]{$a-1$}}
\qbezier(70,10)(70,10)(50,30)
\qbezier(70,30)(70,30)(50,10) 
\put(55,25){\vector(-1,1){5}}
\put(65,25){\vector(1,1){5}}
\put(60,20){\circle{4}}
\put(45,34){\makebox(0,0)[cc]{$b$}}
\put(45,8){\makebox(0,0)[cc]{$a$}}
\put(75,8){\makebox(0,0)[cc]{$b$}}
\put(75,34){\makebox(0,0)[cc]{$a$}}
\end{picture}
\vspace{.5cm}
\caption{Labeling around crossing} \label{Fig1fig3}
\end{figure}

\noindent After labeling, assign a weight $W_{D}(c)$ \cite{kauffman2013affine}, to each classical crossing $c$. For the crossing $c$ in Fig.~\ref{Fig1fig3}, 
$$
W_{D}(c) = \operatorname{sgn} (c)(a-b-1). 
$$ 

\noindent In \cite{cheng2013polynomial}, Z. Cheng and H. Gao assigned an integer value, called \emph{index value}, to each classical crossing $c$ of a virtual knot diagram and denoted it by $\operatorname{Ind}(c)$. 
It was proved \cite[Theorem~3.6]{cheng2013polynomial} that  
\begin{equation}
\operatorname{Ind}(c) = W_{D}(c) = \operatorname{sgn} (c)(a-b-1)  \label{1.eq2.2}
\end{equation}
with $a$ and $b$ be labels as presented in Fig.~\ref{Fig1fig3}. Thus, the index value can be computed using the labeling procedure described in Fig. 3. From equation~(\ref{1.eq2.2}), it is evident that $\abs{\operatorname{Ind}(c)}\leq \abs{\mathcal{C}(D)},$ where $\mathcal{C}(D)$ is the set of classical crossings of $D$ and $|\mathcal{C}(D)|$ is the cardinality of the set $\mathcal{C}(D)$.\\
\noindent In \cite{SatohTaniguchi2014}, S.~Satoh and K.~Taniguchi introduced the $n$-th writhe for virtual knots. For each $n \in \mathbb{Z}\setminus \{0\}$, the $n$-th writhe $J_n(D)$ of an oriented virtual knot diagram $D$ is defined as the number of positive sign crossings minus number of negative sign crossings of $D$ with index value $n$. Note that $J_n (D)$ is  indeed the coefficient of $t^n$ in the affine index polynomial. This $n$-th writhe is a virtual knot invariant, for more details we refer to \cite{SatohTaniguchi2014}. 

\noindent Thus, the index value can be computed using the labeling procedure described in Fig. 3. 
\begin{proposition} \cite[Theorem~1.5]{SatohTaniguchi2014}
\label{prop:boundsUbyS}
Let $K$ be a virtual knot. 
\begin{itemize} 
\item[$(a)$]  If $J_n(K) \neq J_{-n}(K)$ for some $n \in {\mathbb Z}\setminus \{0\}$, then $(1,0) \leq U(K)$. 
\item[$(b)$]  $(0,   \sum_{n \neq 0} | J_n(K) |  /2 ) \leq U(K)$.  
\end{itemize}
\end{proposition}
\noindent A \emph{flat virtual knot diagram} is a virtual knot diagram obtained by forgetting the over/under-information of every real crossing.  It means that a flat virtual knot is an equivalence class of flat virtual knot diagrams by \emph{flat Reidemeister moves} which are Reidemeister moves (see Figs.~\ref{fig1a} and \ref{fig1b}) without the over/under information. We will say that a virtual knot invariant is a \emph{flat virtual knot invariant}, if it is remains unaffected by crossing change operations. For $n\in \mathbb{N}$, a virtual knot invariant, known as {\it $n$-dwrithe}, defined in \cite{kaur2018polynomial} is an example of a flat virtual knot invariant.  

\begin{lemma}\cite{kaur2019unknottingknot}\label{lem}
Let $K$ be a virtual knot and $\overline{K}$ its flat projection. If $\overline{K}$ is non-trivial, then 
$U(K) \geq (1,0)$.  
\end{lemma} 
\begin{remark}\label{lem:flat}
Let $K$ be a virtual knot with $U(K) \geq (1,0)$. Then $K$ cannot
be deformed into the trivial knot by applying the crossing change operation. Thus the flat
projection $\overline{K}$ of $K$ is non-trivial.
By using Lemma~\ref{lem}, we can say the flat projection $\overline{K}$ is non-trivial if and only if
$U(K) \geq (1,0)$.  
\end{remark}

\noindent In \cite{cheng2013polynomial}, Z. Cheng and H. Gao defined an invariant, called span, for 2-component virtual links using Gauss diagram.  Consider a diagram $D=D_1 \cup D_2$ of a virtual link $L =K_1 \cup K_2$. Let us traverse along $D_1$ and consider the linking crossings of $D_1$ and $D_2$. If $\alpha_{+}$ (respectively, $\alpha_{-}$) is the number of over linking crossings with positive sign (respectively, negative sign) and $\beta_{+}$(respectively, $\beta_{-}$) is the number of under linking crossings with positive sign (respectively, negative sign), then $span(D)$ of $D$ is defined as
\[span(D) = | (\alpha_{+} - \alpha_{-})-(\beta_{+}- \beta_{-})|.\]
Further, the definition of span value is extended in \cite{kaur2019unknottinglink} for virtual links having arbitrary components.
Span value of an $n$-component virtual link diagram $D =D_1\cup D_2 \cup \cdots \cup D_n$ is defined as
\[span(D) = \displaystyle \sum_{i\neq j} span(D_i \cup D_j).\]
Since $span(D_i\cup D_j)$ is a virtual link invariant, $span(D)$ is also a virtual link invariant.
\begin{lemma}\cite{kaur2019unknottinglink} \label{lem:span_inv}
 Let $D$ be a virtual link diagram. Then $\operatorname{span}(D)$ is equal to the minimum number of crossings in $D$ which should be virtualized to obtain a diagram $D'$ such that $\operatorname{span}(D')=0$.
\end{lemma}
\noindent Moreover, in \cite[Corollary~2.1]{kaur2019unknottinglink}, it was proved that $U(L)\geq (\operatorname {span}(L),0)$ for any virtual link $L$.
\begin{remark} \label{rem:span_inv}
 By definition, the span value of a virtual link depends only on linking crossings and it is invariant under crossing change operation. Further, the span values of any $n$-component split link and unlink are zero. Let $D$ be a diagram of a virtual link $L=K_1\cup K_2\cup\cdots \cup K_n$ and $U(L)=(n',m')$. Then, \\
\begin{enumerate}
\item by Lemma ~\ref{lem:span_inv}, minimum $\operatorname{span}(D)$ number of linking crossings are required to virtualize in $D$ such that the resulting diagram becomes a diagram of a split link.
\item for any $K_i, ~i\in\{1,2,\ldots,n\}$, if there exists a $k\in {\mathbb Z}\setminus \{0\}$ such that  $J_k(K_i) \neq J_{-k}(K_i)$, then $(1,0) \leq U(K_i)$. That means, at least one virtualization is required in $K_i$ to deform $K_i$ into a trivial knot. 
\end{enumerate}
Since span is a virtual link invariant and remains unaffected by crossing change operations, combining (1) and (2), observe that minimum $\operatorname{span}(D)+1$ number of crossings needed to virtualize in $D$ such that the resulting diagram becomes a diagram of a unlink.  That is $n'\geq (\operatorname{span}(D)+1)$.
\end{remark}

 \noindent The linking number, denoted by $\operatorname{lk(D)}$, of an n-component virtual link diagram $D =D_1\cup D_2 \cup \cdots \cup D_n$ is defined as
\[ \operatorname{lk(D)}=\dfrac{1}{2}\sum_{c\in D_i \cap D_j ,i\neq j}\operatorname{sgn}(c),\]
where the summation runs over all the linking crossings and $\operatorname {sgn}(c)$ is the sign of crossing $c$. The linking number $\operatorname{lk(D)}$ is a virtual link invariant. \\

\noindent Let $S$ be a subset of $\mathcal{C}(D)$ and $|S|$ denote the cardinality of $S$. Then $\Lambda(D)$ and $\ell(D)$ are defined as 
\[
\Lambda(D) =\{ S \subseteq \mathcal{C}(D) \mid  |S| = \operatorname{span}(D) \quad \text{and} \quad  \operatorname{span} (D_{S})=0\},\]
\[\ell_D= \min \{ \abs{ \operatorname{lk}(D_{S})} \mid  S \in \Lambda(D)\},\]
where  $D_{S}$ is a diagram obtained from $D$ by virtualizing all the crossings of $S$.\\ 

\noindent For a virtual link diagram $D$ of 2-components, $\ell_D$ is a virtual link invariant \cite[Proposition~2.1]{kaur2019unknottinglink}. 

\begin{theorem} \cite{kaur2019unknottinglink} \label{thm:bound_links}
If $D=D_{1}\cup D_{2}\cup \ldots \cup D_{n}$ is a diagram of a virtual link $L$, then
\[
\Big(\operatorname{span}(D), \displaystyle \sum_{i \neq j} \ell_{D_{i} \cup D_{j}} + \frac{1}{2} \sum_{i=1}^{n} \sum_{N \in \mathbb{Z}\setminus \{0\}} |  J_N(D_{i}) | \Big) \leq U(L).
\] 
\end{theorem}

\section{Main Results} 
\label{main}
\noindent  In this section, we construct infinite families of virtual knots with unknotting indices $(0,m)$ and $(1,0)$, respectively. However, proving that a virtual knot has an unknotting index of  $(n,m)$ for any pair $(n,m)$ is challenging due to the lack of sufficient stringent lower bounds on the unknotting index of virtual knots. Even demonstrating that a virtual knot has an unknotting index of $(1,m)$ or $(2,0)$ represents a significant challenge. To address this issue, we establish an infinite family of distinct virtual knot diagrams with unknotting indices of $(n,m)$ for any positive integers $n$ and $m$. Moreover,  we prove that these diagrams represent minimal classical crossing diagrams. Furthermore, in Theorem~\ref{thm:(n,m)link} and Proposition~\ref{prop1:(n,m)link}, we positively resolve the scenario for virtual links by establishing the existence of an infinite family of  $k(>1)$-component virtual links with unknotting index $(n,m)$.
 \begin{figure}[!ht] 
{\centering
\includegraphics[scale=0.48]{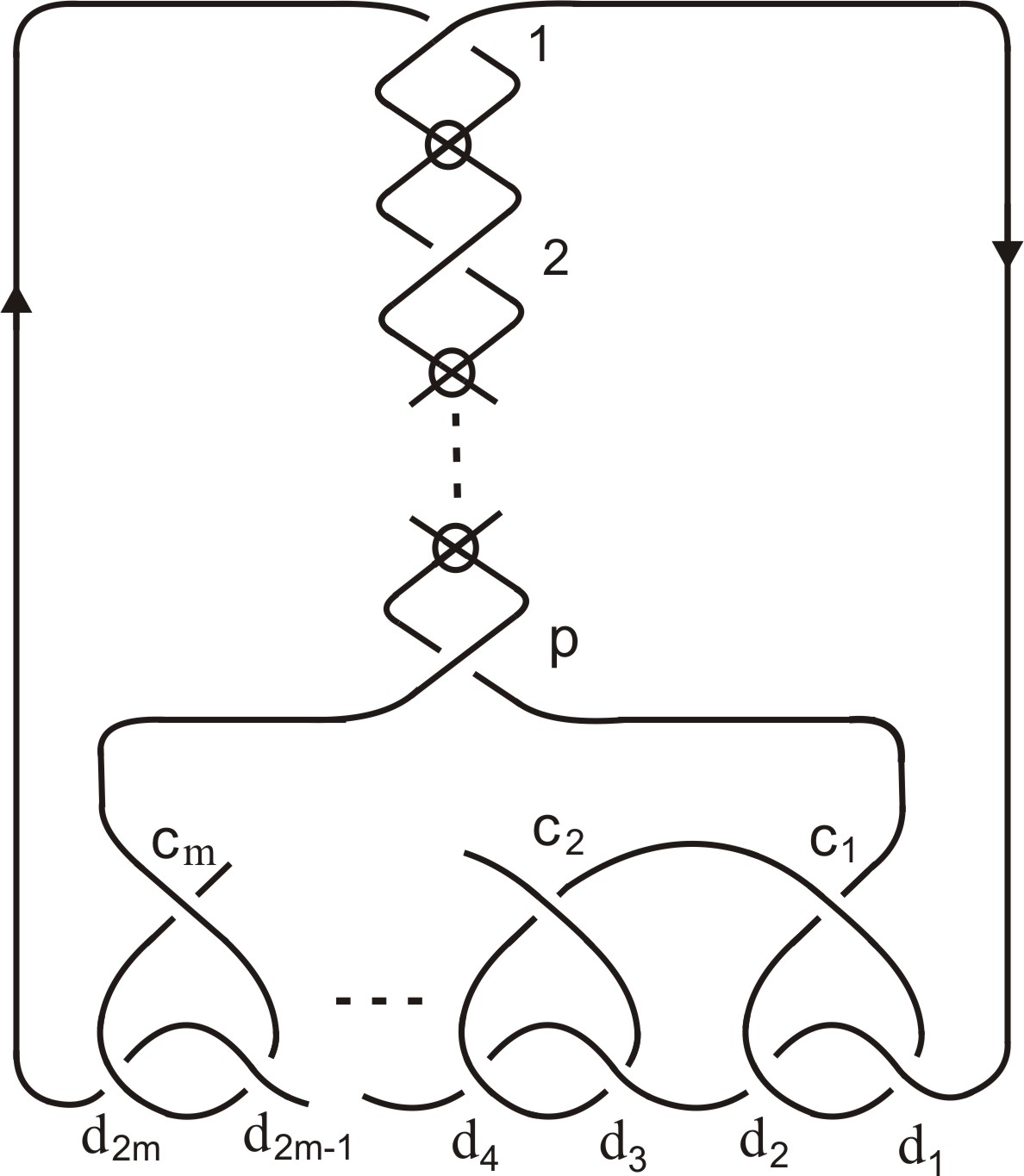}
\caption{Virtual knot diagram $K^{m,p}$. }\label{figrf1}
 }
 \end{figure}
\begin{theorem}
\label{thm:(0,m)}
For any positive integer $m$, there exists an infinite family of virtual knots with  unknotting index $(0,m)$.
\end{theorem}
\begin{proof} To establish this result, we construct  virtual knot diagram $K^{m,p}$ as shown in Fig.~\ref{figrf1}, where $m$, and $p\geq 2$ are positive integers.  Let  $c_{1}, c_{2},\ldots,c_{m}$ and $d_{1},\ldots, d_{2m}$ denote the classical crossings as shown in the diagram $K^{m,p}$. Then the index values of these crossings are given as
\[ \operatorname{Ind}(c_{i})=0, \text{~for~} i=1,2,\ldots, m,\]
\[\operatorname{Ind}(d_{i})=\begin{cases}
 p-1,&  \text{for~} i=1,3,\ldots, 2m-1,  \\
 -(p-1),          & \text{for~} i=2,4,\ldots, 2m.
             \end{cases}\] 
The index value of the classical crossings in the vertical block of $K^{m,p}$  is
\[\operatorname{Ind}(1)=\operatorname{Ind}(2)= \ldots = \operatorname{Ind}(p)=0.\] 
\noindent The sign value of all the classical crossings in $K^{m,p}$ except $c_{i}$ crossings is $(-1)$ and $\operatorname{sgn}(c_i)=+1$ for $i=1,2,\ldots, m.$ Moreover, $p\geq 2$ implies $1-p\neq 0$. Therefore, the writhe invariants $J_{(p-1)}$ and $J_{(1-p)}$ are given by
\[J_{(p-1)}(K^{m,p})= \displaystyle \sum_{i=1}^{m}\operatorname{sgn}(d_{2i-1})= -m,  ~ J_{-(p-1)}(K^{m,p})= \displaystyle \sum_{i=1}^{m}\operatorname{sgn}(d_{2i})= -m,\text{~and}\] 
\begin{equation}\label{Eq0(0,m)}
 \sum_{k \neq 0} | J_k(K^{m,p}) |= \abs{J_{(p-1)}(K^{m,p})}+ \abs{J_{(1-p)}(K^{m,p})}=2m.
\end{equation}
By using Proposition~\ref{prop:boundsUbyS}$(b)$, and equation~(\ref{Eq0(0,m)}), we have 
\begin{equation}\label{Eq1(0,m)}
 U(K^{m,p})\geq \Big(0, \dfrac{1}{2}\sum_{k \neq 0} | J_k(K^{m,p}) | \Big)=(0,m).
 \end{equation}
\noindent Also by applying crossing change operations at $d_2, d_4,\ldots, d_{2m}$ crossings, the resulting diagram becomes a trivial knot diagram as shown in Fig.~\ref{figrf3}. Therefore,
\begin{equation}\label{Eq2(0,m)}
U(K^{m,p})\leq (0,m).
\end{equation}
Evidently, equation~(\ref{Eq1(0,m)}), and equation~(\ref{Eq2(0,m)}) implies  $U(K^{m,p})= (0,m)$ .\\
Observe that for different integers $p_1, p_2\geq 2$, 
\[J_{p_1-1}(K^{m,p_1})=J_{p_2-1}(K^{m,p_2})=-m,\text{~and~} J_{p_1-1}(K^{m,p_2})=J_{p_2-1}(K^{m,p_1})=0.\]
\noindent Since the $n$-th writhe is a virtual knot invariant, $K^{m,p_1}$ and $K^{m,p_2}$ represent diagrams of two  different  virtual knots. Therefore, for a fixed value $m$, there are infinitely many positive integers $p\geq 2$ such that $U(K^{m,p})=(0,m)$. Hence the result follows.
 \end{proof}

 \begin{figure}[!ht] 
{\centering
\includegraphics[scale=0.4]{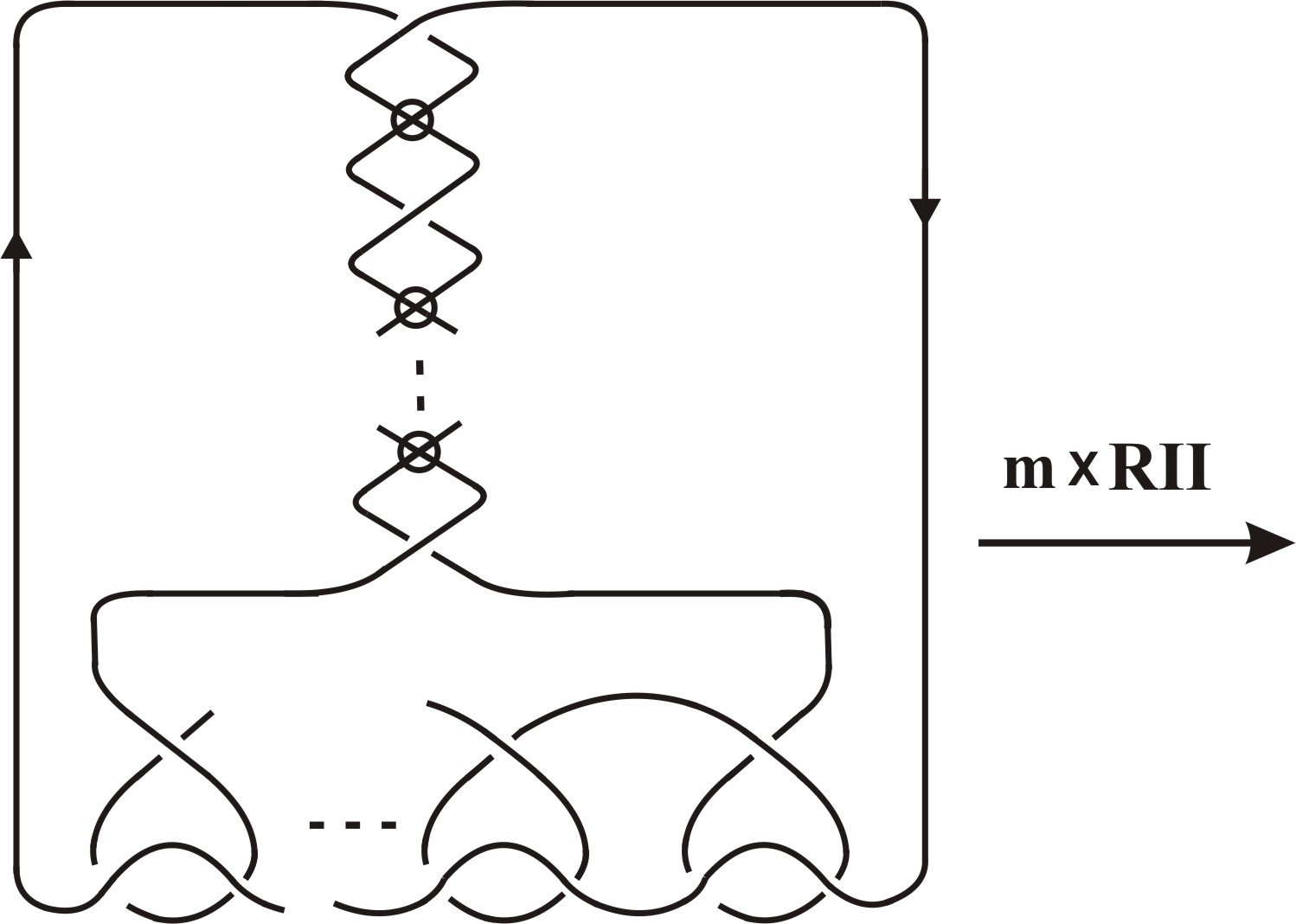} \includegraphics[scale=0.4]{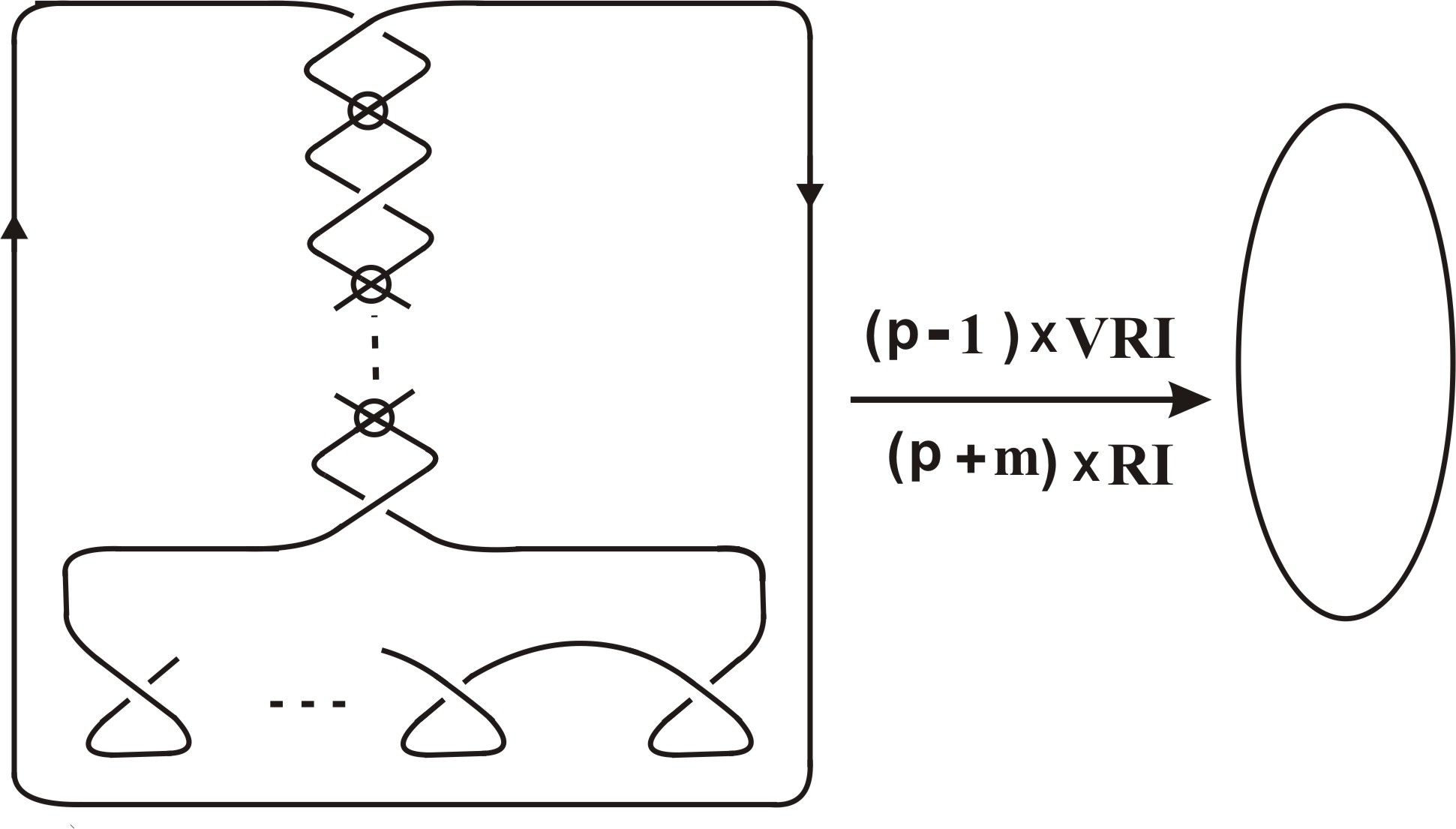}
 \caption{A diagram of the trivial knot.}
\label{figrf3}}
\end{figure}

\begin{figure}[!ht] 
{\centering
\subfigure[$D^{n,p}$]{\includegraphics[scale=0.45]{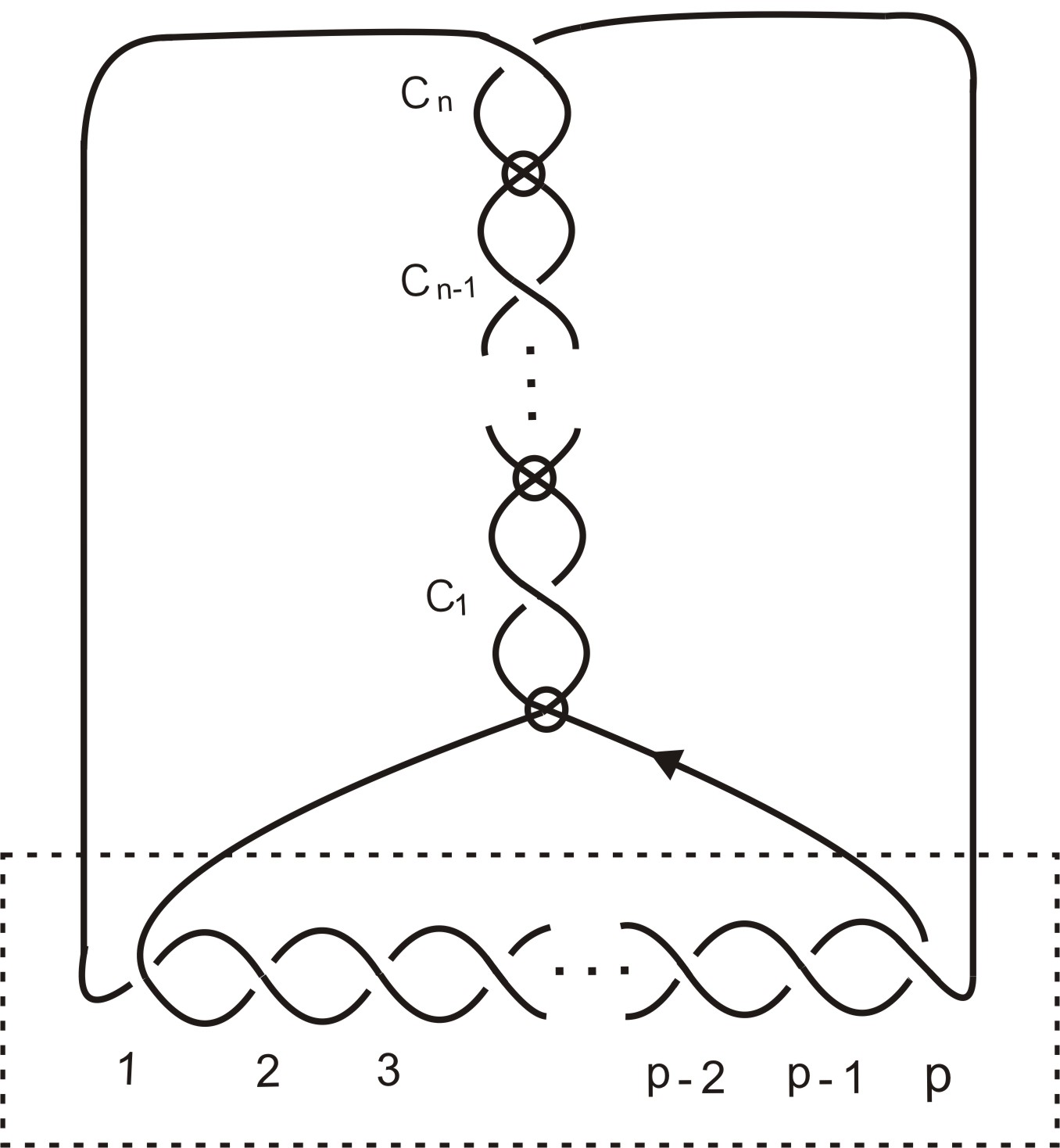}\label{tw1}}
  \hspace{.5cm}
 \subfigure[$\tilde{D}^{n,p}$]{\includegraphics[scale=0.45]{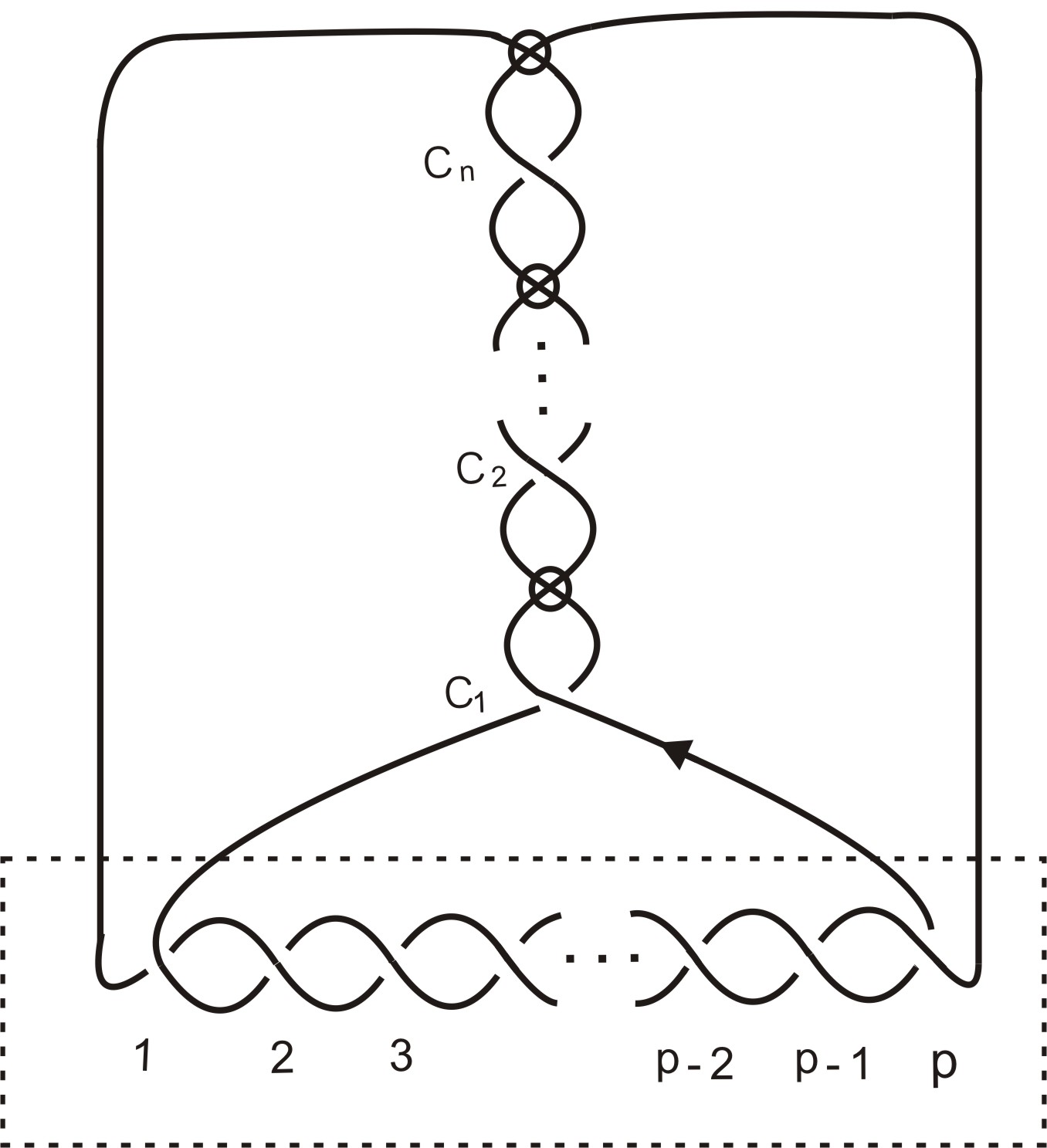} \label{tw2} }
\caption{Virtual knot diagrams $D^{n,p}$ and $\tilde{D}^{n,p}$. }
\label{ttw}
 }
 \end{figure}
 
\noindent In order to establish  Theorems~\ref{thm:(1,0)knot} and \ref{thm:(n,m)knots},  we first  construct a class of virtual knot diagrams and define few notations.
\begin{definition} {\rm 

A diagram $D$ of a virtual knot $K$ is said to have  {\it minimum classical crossings} if, 
\[|\mathcal{C}(D)|=\min\Big\{|\mathcal{C}(D')| \Big | ~D'\in [K]\Big\},\]
where $\mathcal{C}(D)$ is set of classical crossing of $D$ and $|\mathcal{C}(D)|$ is the cardinality of the set $\mathcal{C}(D)$. Then a diagram $D$ of $K$ with {\it minimum number of classical crossings} is called a minimal classical crossing diagram.
}\end{definition}
\noindent Construct virtual knot diagrams denoted by $D^{n,p}$ and  $\tilde{D}^{n,p}$ with $n$ number of virtual crossings as shown in Fig.~\ref{tw1} and \ref{tw2}, respectively.  Here, we discuss the unknotting index for a special class of virtual knots that are obtained by virtualizing some $m$ number of classical crossings in $D^{n,p}$ and  $\tilde{D}^{n,p}$. In particular, due to similarity in $D^{n,p}$ and  $\tilde{D}^{n,p}$, we discuss only the class of virtual knots that are obtained from $D^{n,p}$. We label the crossings of $D^{n,p}$ and  $\tilde{D}^{n,p}$ that are shown in the respective rectangular boxes in Fig.~\ref{ttw} with integers. We denote these crossings as horizontal block crossings. For convenience, whenever we virtualize some crossings in $D^{n,p}$ and  $\tilde{D}^{n,p}$, labeling in the resulting diagram remains same. We denote $D_{q,r}^{n,p}$ as a virtual knot diagram  obtained form $D^{n,p}$ by virtualizing $q$ crossings at odd labelings and $r$ crossings at even labelings in the horizontal block.

\begin{lemma}
\label{lem-(n,m)}
If  $D$ is a virtual knot diagram $D^{n,p}_{q,r}$ with odd $p$, then  whenever $q-r>2$ or $q-r <0$, there exist at least one $k\in \mathbb{Z}\setminus\{0\}$ such that $J_{k}(D) =-1$ and  $J_{-k}(D)=0 $.
\end{lemma}
\begin{proof} Let $D$ be a virtual knot diagram $D^{n,p}_{q,r}$, where $p$ is an odd natural number. Then the index value of the classical crossings in the horizontal block of $D$ corresponds to odd and even labelings is $(-n)$ and $n$, respectively. Let $c_{1}, c_{2},\ldots,c_{n}$ denote the classical crossings of $D$ corresponding to $c_i$ crossings of $D^{n,p}$ as shown in Fig.~\ref{tw1}. Then 
 \[ \operatorname{Ind}(c_{i}) = (2i-n-q+r),\text{~for~} i=1,2,\ldots, n, \text{~and~} \]
 \[  \operatorname{Ind}(c_{i}) = -\operatorname{Ind}(c_{j}) \text{~if and only if~} i+j=n+q-r.\] 
\noindent It is evident that all the  $c_i$ crossings in $D$ have distinct index values. \\
If $q-r>2$, then $\operatorname{Ind}(c_{1})=2-n-(q-r)<-n$.\\
 Assume that there exist some $i\in \{2,3, \ldots,n\}$ such that $\operatorname{Ind}(c_{1})=-\operatorname{Ind}(c_{i})$. This implies that $i=n+q-r-1.$ As $q-r>2$, we have $i=n+q-r-1>n+1$, which is a contradiction to the fact that $i\leq n$. Consequently, there is no crossing other than $c_1$ in $D$, whose index value is $\pm \operatorname{Ind}(c_1$). Thus, for $k=\operatorname{Ind}(c_{1})$, we have
\[J_{k}(D) =\operatorname{sgn}(c_1)=-1 \text{~and~}  J_{-k}(D)=0. \]
If $q-r<0$, then Ind$(c_{n})=n-(q-r)>n$.  \\               
Assume that there exist some $i\in \{2,3, \ldots,n\}$ such that $\operatorname{Ind}(c_{n})=-\operatorname{Ind}(c_{i})$. This implies that $i=q-r<0,$ which is a contradiction to the fact that $i\geq 1$. Consequently, there is no crossing other than $c_n$ in $D$, whose  index value is $\pm \operatorname{Ind}(c_n$). Therefore, for $k=\text{Ind}(c_{n})$, we have
\[J_{k}(D) =\operatorname{sgn}(c_n)=-1 \text{~and~}  J_{-k}(D)=0. \]
Hence the desired result follows.
\end{proof}
\begin{lemma}\label{Mlemma}
If a diagram $D$ of a virtual knot $K$ satisfies the following two conditions:
\begin{enumerate}
\item[$(a)$] whenever $\operatorname{Ind}(c_i)=\operatorname{Ind}(c_j)$, then $\operatorname{sgn}(c_i)=\operatorname{sgn}(c_j)$, for any $c_i,c_j\in\mathcal{C}(D) $, and
\item[$(b)$] $\operatorname{Ind}(c)\neq 0$, for any  $c\in\mathcal{C}(D) $,
\end{enumerate}
then $D$ is a diagram of $K$ with minimum number of classical crossings. 
\end{lemma}
\begin{proof}
Let $D$ be a diagram of a virtual knot $K$ satisfying the two conditions given in the hypothesis. To establish that $D$ is a diagram with minimum number of classical crossings, it is enough to show that $\displaystyle \sum_{k}\abs{J_k(D)}=\abs{\mathcal{C}(D)}$.\\
Let  $C_k(D)=\{c\in \mathcal{C}(D)| {\rm Ind}(c)=k\}$, for an integer $k\neq0$.
 Since number of crossings in $D$ are finite, there  exist $N\in \mathbb{N}$ such that $\abs{\operatorname{Ind}(c)}\leq N$, for all $c\in \mathcal{C}(D).$\\

\noindent This implies that there exists only finitely many $C_k(D)$, which are subsets of $\mathcal{C}(D)$. Since index value of a crossing is unique and there is no crossing with index zero, 
\begin{equation}\label{Eqq01}
C_k(D) \cap C_l(D) = \emptyset, \text{~for~} k \neq l \quad \text{~and hence~}\quad \mathcal{C}(D)=\displaystyle\cup_{k} C_k(D).\end{equation}
\noindent By hypothesis, since crossings of same {\it index} have same sign in $D$, we have  
\begin{equation}\label{Eqq02}\abs{J_k(D)}=\abs{ \displaystyle \sum_{\operatorname{Ind}(c)=k} \operatorname{sgn}(c)~~}=|C_k(D)|.
\end{equation}  
Equations~(\ref{Eqq01}) and (\ref{Eqq02}) implies that
\begin{equation}\label{Eqq1}
\displaystyle \sum_{k}\abs{J_k(D)}= \displaystyle \sum_{k} \abs{C_k(D)}=|\mathcal{C}(D)|.
\end{equation}

\noindent Observe that, if $D'$ is any diagram that is equivalent to $D$ with $m$ number of classical crossings, then
\begin{equation}\label{Eqq2}
\begin{split}
\displaystyle \sum_{k \neq 0}\abs{J_k(D')} &= \displaystyle \sum_{k \neq 0}~\left(\hspace{.1cm}\abs{ ~\displaystyle \sum_{~\mathclap{\substack{\operatorname{Ind}(c)=k\\ c\in \mathcal{C}(D')}}} \operatorname{sgn}(c)~~}\right)\leq \displaystyle \sum_{k \neq 0}~\left(\hspace{.4cm}\displaystyle \sum_{\mathclap{\substack{\operatorname{Ind}(c)=k\\ c\in \mathcal{C}(D')}}} \abs{\operatorname{sgn}(c)}\right)\\
&= \displaystyle \sum_{k \neq 0}\abs{C_k(D')}\leq \abs{\mathcal{C}(D')}=m,
\end{split}
\end{equation}
since the index value of each crossing is unique.
\noindent As the $n$-th writhe is a virtual knot invariant, from equation~ (\ref{Eqq1}) and equation~(\ref{Eqq2}), we have
\[|\mathcal{C}(D)|=\displaystyle \sum_{k }\abs{J_k(D)}=\displaystyle \sum_{k \neq 0}\abs{J_k(D')}\leq m.\]
Hence the  number of crossings in any diagram which is equivalent to $D$, are at least $|\mathcal{C}(D)|$. Thus $D$ is a diagram of $K$  with minimum number of classical crossings. Hence the result follows.
\end{proof}

\begin{theorem}\label{thm:(1,0)knot}
There exist infinitely many virtual knots with unknotting index $(1,0)$. 
\end{theorem}
\begin{proof}

Let $p$ and $q$ be any two positive integers with odd $p$ satisfying $p>q>3$.  
To establish this result, we construct a diagram $D$ obtained from the virtual knot diagram $D^{1,p}$ as shown in Fig.~\ref{tw1}, by virtualizing $q$ crossings with odd labelings in the horizontal block. It is evident that $D$ is the diagram $D^{1,p}_{q,r}$ with $r=0$. Let $K$  be the virtual knot presented by the diagram $D$ and by hypothesis $q-r=q>3$. Therefore, by using Lemma~\ref{lem-(n,m)} and Proposition~\ref{prop:boundsUbyS}(a), we have
\begin{equation}\label{Eq1(1,0)}
U(K)\geq (1,0).
\end{equation}

\begin{figure}[!ht] 
{\centering
 \includegraphics[scale=0.5]{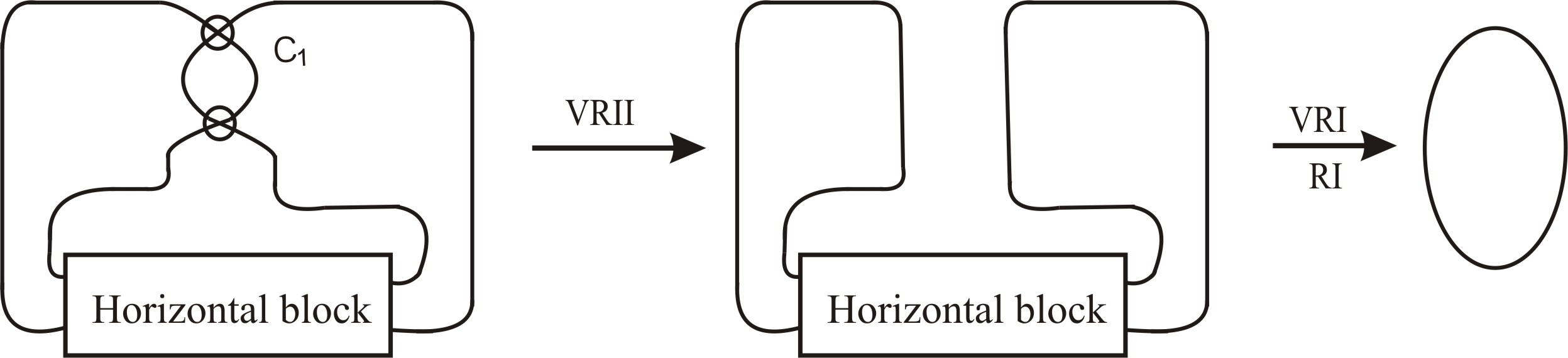} 
  \caption{A diagram of the trivial knot. }\label{fig10}
 } \end{figure}
   
\noindent For upper bound, observe that by virtualizing crossing $c_1$ in the diagram $D^{1,p}_{q,0}$ the resulting diagram becomes a diagram of the trivial knot as shown in Fig.~\ref{fig10}. Hence 
\begin{equation}\label{Eq2(1,0)}
U(K)\leq U(D)\leq (1,0).
\end{equation}
By equation~(\ref{Eq1(1,0)}) and equation~(\ref{Eq2(1,0)}), we have $U(K)=(1,0)$.

\noindent It is evident from the proof of Lemma~\ref{lem-(n,m)} that the index values of the classical crossings in the horizontal block of $D(=D^{1,p}_{q,0})$ corresponds to odd and even labeling are $(-1)$ and 1, respectively. Moreover, the index value of crossing $c_1$ in the diagram $D$ is $\operatorname{Ind}(c_{1}) = (1-q)$, and sign of all the crossings of $D$ is $(-1)$.\\
\noindent Thus,  the diagram $D(=D^{1,p}_{q,0})$ has no crossing with zero index and crossings with the same index values having the same sign value. By Lemma~\ref{Mlemma}, $D$ is a diagram of $K$ with minimum number of classical crossings. The total number of classical crossings in $D$ are $p-q+1$. \\
Consider two odd positive integers $p_1\neq p_2$ and an integer $q$, satisfying
 \[ p_1>q>3 \quad \text{and}\quad p_2>q>3.\]
Let $K_{i}$ be the virtual knot corresponding to the diagram $D^{1,p_i}_{q,0}$, $i\in\{1,2\}$. Clearly, the unknotting index of $K_1$ and $K_2$ is $(1,0)$. Moreover, $D^{1,p_i}_{q,0}$ is a diagram of $K_i,~i\in\{1,2\}$ with  minimum number of classical crossings and
\[|\mathcal{C}(D^{1,p_1}_{q,0})|=p_1-q+1\neq p_2-q+1=|\mathcal{C}(D^{1,p_2}_{q,0})|.\] 
Therefore, the virtual knots corresponding to $D^{1,p_1}_{q,0}$ and $D^{1,p_2}_{q,0}$ are non equivalent. 
Moreover, there exists infinitely many odd integers $p$ satisfying $p>q$ for any integer $q>3$. Thus, there exists infinitely many virtual knots with unknotting index $(1,0)$. Therefore, the result holds.
\end{proof}

\begin{proposition}
\label{thm-(n,m)}
If $D$ is a diagram of a virtual knot $K$, that is obtained by virtualizing $m=q+r$ number of classical crossings in the horizontal block of $D^{n,p}$, see  Fig.~\ref{tw1}, where $p$ is an odd positive integer, and $q$ and $r$ are the number of crossings virtualized at odd and even labelings, respectively, then 

\[U(D)=\begin{cases}
\Big(r-q,\Big \lfloor \dfrac{p+n-2r}{2}\Big\rfloor \Big),          &\text{if}~~-n<q-r<0,\\

\Big(0, \Big\lfloor \dfrac{p+n-m}{2}\Big\rfloor\Big),&  \text{if}~~0\leq q-r\leq 2,  \\

\Big (q-r-2,\Big \lfloor \dfrac{p+n-2q+2}{2}\Big\rfloor\Big),          &\text{if}~~2<q-r< n+2,\\ 
   (n,0),          & \text{ Otherwise. }
   \end{cases}\] 
 Moreover, for $0\leq q-r\leq 2$, $U(D)=U(K).$
  \end{proposition}
\begin{proof} Let $D$ be a virtual knot diagram of a virtual knot $K$ that is obtained  by virtualizing $q+r$ number of classical crossings in the horizontal block of $D^{n,p}$ with odd $p$ as shown in  Fig.~\ref{tw1}, where $q$ and $r$ are the number of  crossings virtualized at odd and even labelings, respectively. Observe that $D$ is exactly the virtual knot diagram $D_{q,r}^{n,p}$.   Let $c_{1}, c_{2},\ldots,c_{n}$ denote the classical crossings of $D$ corresponding to $c_i$ crossings of $D^{n,p}$ as shown in Fig.~\ref{tw1}.  Then  \[ \operatorname{Ind}(c_{i})=(2i-n-q+r), \text{~for~} i=1,2,\ldots,n,\quad \text{and},\]
\[ \operatorname{Ind}(c_{i}) = -\operatorname{Ind}(c_{j}) \text{~if and only if~} i+j=n+q-r.\] 
The index value of the classical crossings in the horizontal block of $D$ corresponding to odd and even labelings is $(-n)$ and $n$, respectively.\\
 With these observations, we first prove the result for  $0\leq q-r \leq 2$. \\
\noindent The total number of classical crossings in $D$ are $p+n-m.$ When $p+n-m$ is even, then all classical crossings in $D$ have non-zero {\it index}. Otherwise, there exist only one crossing in $D$ whose {\it index} value is zero. Since sign of each crossing in $D$ is $-1$, 
\[ \displaystyle \sum_{k \neq 0}\abs{J_{k}(D)/2}= \Big \lfloor \frac{p+n-m}{2}\Big \rfloor .\]
 Hence by using Proposition~\ref{prop:boundsUbyS}$(b)$, we have 
\begin{equation}\label{eq:thm(n,m)1}
U(K)\geq \Big(0, \Big\lfloor\frac{p+n-m}{2} \Big\rfloor\Big).
\end{equation} 
\noindent Apply crossing change operation at all the crossings in the horizontal block labeled with even integers and at $c_{n}, c_{n-1}, \ldots, c_{\lfloor\frac{n+2}{2}\rfloor}$, if $q-r=0$, or 
  at $c_{1}, c_{2}, \ldots, c_{\lfloor \frac{n}{2}\rfloor}$, if $q-r=1$, or at $c_{n}, c_{n-1}, \ldots, c_{\lfloor\frac{n+4}{2}\rfloor}$, if $q-r=2$. Then the resulting diagram is a diagram of the trivial knot and the total number of crossing changes is equal to $ \big \lfloor\frac{p+n-m}{2}\big\rfloor .$
Hence 
\begin{equation}\label{eq:thm(n,m)1b}
U(K)\leq \Big(0, \Big\lfloor \dfrac{p+n-m}{2}\Big\rfloor\Big).
\end{equation} 
Equation~(\ref{eq:thm(n,m)1}) and equation~(\ref{eq:thm(n,m)1b}) implies that,
\[U(K)=\Big(0, \Big\lfloor \dfrac{p+n-m}{2}\Big\rfloor\Big).\]

 \noindent When $2<q-r<n+2$, it is evident that by Lemma~\ref{lem-(n,m)} and Proposition~\ref{prop:boundsUbyS}$(a)$, we have $U(K)\geq (1,0)$. Hence virtualization is required to deform $D$ into the trivial knot.  

\noindent Let $D'$ be the diagram obtained from $D$ by virtualizing $t=n_1+q_1+r_1$ number of crossings, where $n_1, q_1$ and $r_1$ be the number of $c_i's$, odd labeling and even labeling crossings that are virtualized in $D$. Then $D'$ is  equivalent to the virtual knot diagram $D^{n',p}_{q',r'}$, where $n'=n-n_1, q'=q+q_1$ and $r'=r+r_1$.

\noindent If $t<q-r-2$, then $n_1+2q_1+2<(q+q_1)-(r+r_1)$, i.e., $2<q'-r'$. If $t=q-r-2$ and $q'-r'\neq 2$, then also $q'-r'>2$.\\
By Lemma~\ref{lem-(n,m)}, there exist at least one $k$ such that 
$J_{k}(D')=-1$ and  $J_{-k}(D')=0 $. From Proposition~\ref{prop:boundsUbyS}$(a)$ and Remark~\ref{lem:flat}, the flat projection of $D'$ is non-trivial.  
In case $t=q-r-2$ and $q'-r'=2$, there are $p+n-2q+2$ number of classical crossings in $D'$ and 
\[ \displaystyle \sum_{k \neq 0}\abs{J_{k}(D')/2}= \Big \lfloor \frac{p+n-2q+2}{2}\Big \rfloor \text{~ (as in the case~} 0 \leq q-r\leq 2).\] Therefore, we have
\begin{equation}\label{eq:thm(n,m)2}
 U(D)\geq \Big(q-r-2, \Big\lfloor \frac{p+n-2q+2}{2}\Big\rfloor\Big).
\end{equation}
\noindent For the reverse inequality,  virtualize $q-r-2$ number of crossings labeled with even integers in the horizontal block of $D$. Then in the resulting diagram $D'$, we have $q'-r' =2$, and hence $D'$ satisfies the condition $0 \leq q'-r'\leq 2$. Therefore, $ U(D')= \Big(0, \Big\lfloor \dfrac{p+n-2q+2}{2}\Big\rfloor\Big)$ and hence
\begin{equation}\label{eq:thm(n,m)3}U(K)\leq U(D)\leq \Big(q-r-2, \Big\lfloor \dfrac{p+n-2q+2}{2}\Big\rfloor\Big).
\end{equation} 
Equation~(\ref{eq:thm(n,m)2}) and equation~(\ref{eq:thm(n,m)3}) implies
\[U(D)= \Big(q-r-2, \Big\lfloor \dfrac{p+n-2q+2}{2}\Big\rfloor\Big).\]

\noindent In the case when $q-r\geq n+2$ or $q-r \leq -n$, observe that either $q-r>2$ or $q-r < 0$.\\
 By Lemma~\ref{lem-(n,m)} and Proposition~\ref{prop:boundsUbyS}$(a)$, we have $U(K)\geq (1,0)$. 

\noindent Observe that by virtualizing $c_{1},c_{2}, \ldots,c_{n}$ crossings in $D$, the resulting diagram becomes a diagram of the trivial knot. Thus 
\begin{equation}\label{eq:thm(n,m)4}
U(K)\leq U(D)\leq(n,0).\end{equation}

\noindent Now we prove that, by virtualizing less than $n$ number of crossings in $D$, the flat virtual knot corresponding to the resulting virtual knot is always non-trivial.

\noindent Let $D'$ be the diagram obtained from $D$ by virtualizing $t=n_1+q_1+r_1$ number of crossings, where $n_1, q_1,$ and $r_1$ be the number of $c_i's$, odd labeling and even labeling crossings that are virtualized in $D$, respectively. Then $D'$ is equivalent to the virtual knot diagram  $D^{n',p}_{q',r'}$, where $n'=n-n_1, q'=q+q_1$ and $r'=r+r_1$.\\
 If $q-r\geq n+2$ and $t<n$, then $t<n\leq q-r-2$ and hence $2< q'-r'$. \\
  If $q-r\leq -n$ and $t<n$, then $t<n\leq r-q$ and hence $ q'-r'<0$. In both the cases, by Lemma~\ref{lem-(n,m)} and Proposition~\ref{prop:boundsUbyS}$(a)$, the flat projection of $D'$ is non-trivial. Therefore, {\bf $U(D)\geq (n, 0)$} and hence by equation~(\ref{eq:thm(n,m)4}), we have
\[ U(D)= (n, 0).\]   

\noindent In case when $-n<q-r<0$, we use  similar arguments as given the case of $2<q-r<n+2$ of Theorem~\ref{thm-(n,m)}, to conclude
   $$ U(D)= (r-q,\big \lfloor \frac{p+n-2r}{2}\big\rfloor).$$
 Hence the proof of the theorem follows.
\end{proof}

\begin{remark}{\rm
If $D$ is a virtual knot diagram obtained from $D^{n,p}$ (or $\tilde{D}^{n-1,p}$) by virtualizing some $c_{j},~1\leq j\leq n$, crossing, then $D$ is equivalent to $D^{n-1}$ (or $\tilde{D}^{n-1,p}$). Thus, using Proposition~\ref{thm-(n,m)}, we can find  unknotting index of any virtual knot diagram obtained by virtualizing some crossings in $D^{n,p}$ (or $\tilde{D}^{n,p}$).} 
\end{remark}

\begin{figure}[!ht] 
{\centering
\subfigure[Virtual knot diagram $D^{n,p,m}$.]{\includegraphics[scale=0.5]{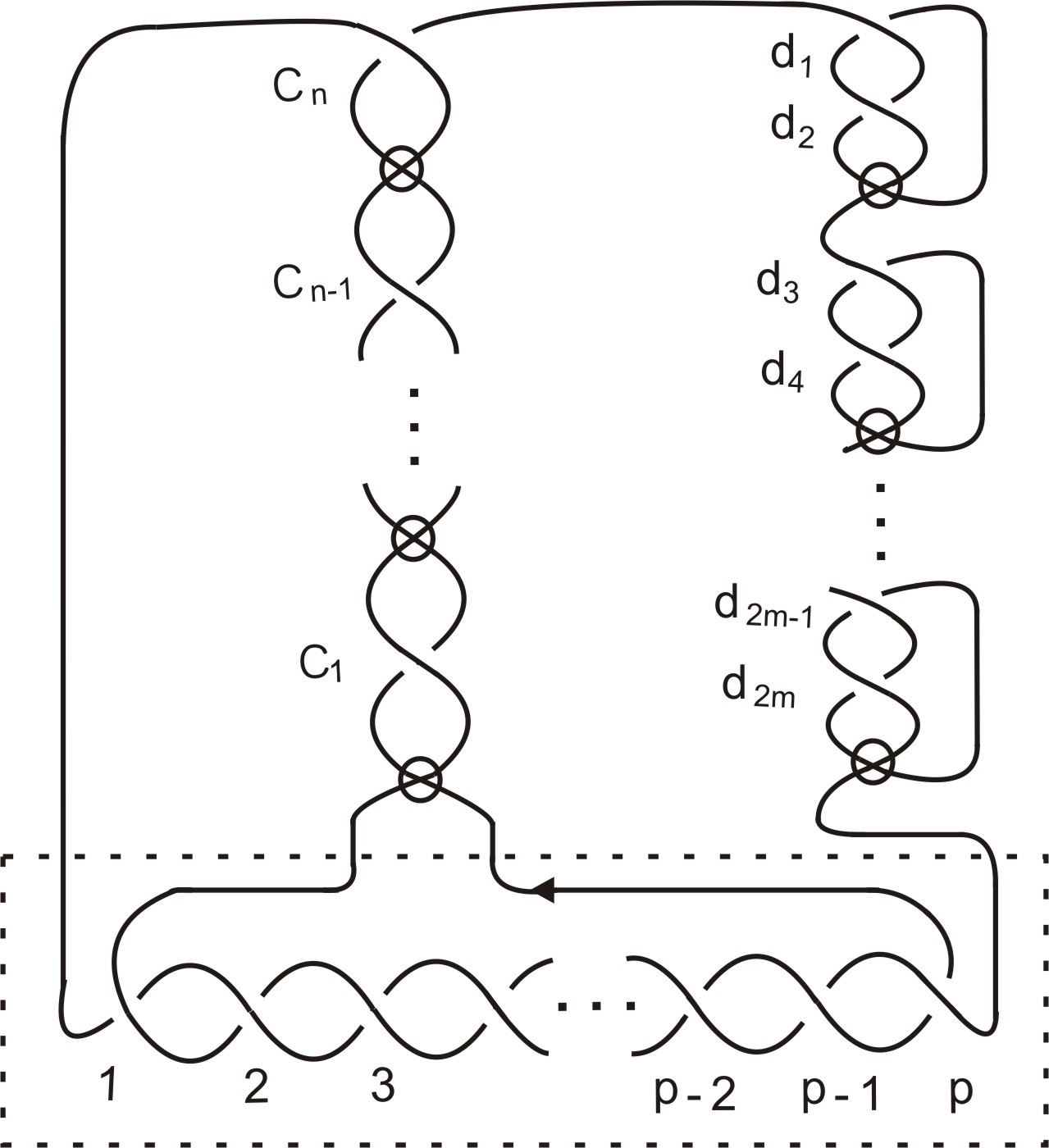}\label{mnu}}
  \hspace{2cm}
 \subfigure[Virtual knot Diagram.]{ \includegraphics[scale=0.6]{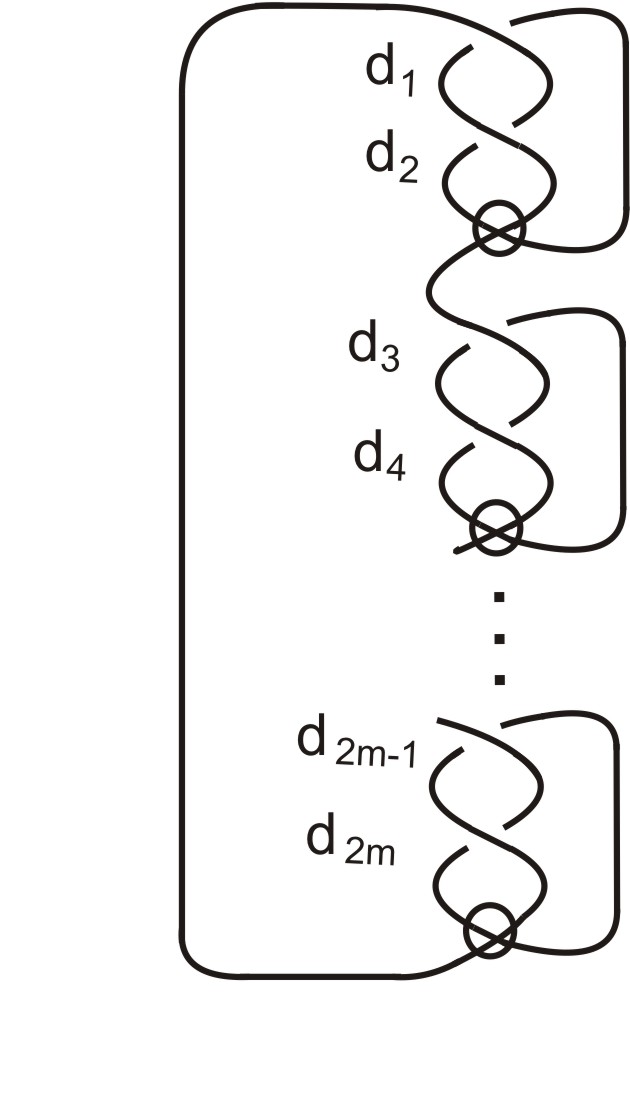}\hspace{.5cm} \label{mnu2} }
\caption{ }
 } \end{figure}
 
\noindent Before proving the Theorem ~\ref{thm:(n,m)knots}, we define $D^{n,p,m}_{q,r,t}$ as a virtual knot diagram obtained from the diagram $D^{n,p,m}$ as shown in Fig.~\ref{mnu} by virtualizing $q+r+t$ number of crossings, where $t$, $q$, $r$ are the number of $d_i's$ crossing, odd and even labeling crossings in the horizontal block, respectively. It is evident that for $m=0$,  $D^{n,p,m}$ is the virtual knot diagram $D^{n,p}$ as shown in Fig.~\ref{ttw}.
 
\begin{theorem}\label{thm:(n,m)knots}
For any $n,m \in \mathbb{N}$, there exists an infinite family of distinct virtual knot diagrams with unknotting index $(n,m)$.
\end{theorem}

\begin{proof}
Let $n,m\in \mathbb{N}$. Consider $q\in \mathbb{N}$ and $p$ be any odd integer such that $p>q>n+2$. Construct a virtual knot diagram $D$ from the diagram $D^{n,p,m}$ given in Fig.~\ref{mnu}, by virtualizing $q$ crossings with odd labelings in the horizontal block.  Assume that $U(D)=(n_D,m_D)$. For our convenience, labeling of crossing in $D$ remains same as in the diagram $D^{n,p,m}$ shown in Fig.~\ref{mnu}.  Then the {\it index values} of the crossings $c_1, c_2,\ldots, c_n$ and $d_1, d_2,\ldots, d_{2m}$  in the diagram $D$ are given as,
\[{\rm Ind}(c_i)=2i-n-q, {\rm ~for~} i=1,2, \ldots,n, \]
\[\quad {\rm and} \quad {\rm Ind}(d_j)=(-1)^{j+1}, {\rm ~for~} j=1,2,\ldots, 2m.\]
The {\it index values} of the crossings of $D$ in the horizontal block corresponding to odd and even labelings are $(-n)$ and $n$, respectively.
\noindent It is evident that the flat projection $\overline{D}$ of $D$ is equivalent to the flat projection of the diagram $D_{q,0}^{n,p}$, which is non-trivial by Lemma~\ref{lem-(n,m)}, Proposition \ref{prop:boundsUbyS}(a) and Remark \ref{lem:flat}. Hence $\overline{D}$ is non-trivial and by Remark~\ref{lem:flat}, $U(K)>(1,0)$, where $K$ is the virtual knot corresponding to diagram $D$. Therefore, virtualization is required to deform $D$ to a trivial knot diagram.

\noindent Let $K'$ be the virtual knot presented by a diagram $D'$ which is obtained from $D$ by virtualizing $s=n_1+t+q_1+r_1$ number of crossings, where $n_1, t, q_1 $ and $r_1$ be the number of $c_i's$,  $d_i's$, odd labeling and even labeling crossings that are virtualized in $D$, respectively. 

\noindent If  $n_1<n$ and $s \leq n$, then $-r_1\geq n_1+q_1+t-n$. As $q > n+2$ is given, 
\begin{align*}
(q+q_1)-r_1 & > (n+2)+q_1-r_1 \\
& \geq (n+2)+q_1+(n_1+q_1+t-n)\\
& =2+2q_1+t+n_1\geq 2. 
\end{align*}
It is evident that the flat projection $\overline{D'}$ of $D'$ is equivalent to the flat projection of $D^{n', p}_{q', r'}$ with $q'-r'> 2$, where $q'=q+q_1, r'=r_1$ and $n'=n+n_1$. By Lemma~\ref{lem-(n,m)}, Proposition~\ref{prop:boundsUbyS} and Remark~\ref{lem:flat}, the flat projection of $D^{n',p}_{q',r'}$ is non-trivial and hence $\overline{D'}$ is a non-trivial flat projection. Again by Lemma~\ref{lem}, $U(K')\geq(1,0)$  which implies $n_D>s$ whenever $n_1<n$ and $s \leq n$.
  
\noindent  When $n_1=n=s$, the diagram $D'$ is equivalent to the diagram shown in Fig.~\ref{mnu2}. Thus, the flat projection of $\overline{D'}$ of $D'$ is equivalent to the flat projection of the diagram shown in Fig.~\ref{mnu2} which is trivial. Hence $\overline{D'}$ is a trivial flat projection.\\
\noindent Thus, minimum $n$ number of virtualizations are required in $D$ such that the flat projection of the resulting diagram is trivial. Therefore, $n_D=n$ and $D'$ is our desired diagram obtained from $D$ by virtualizing $n$ crossings with $n_1=n$. Since diagram $D'$ is equivalent to the diagram shown in Fig.~\ref{mnu2},
$\displaystyle \sum_{k \neq 0}\abs{J_{k}(D')}= 2m$. 

\noindent By Remark~\ref{lem:flat} and Proposition~\ref{prop:boundsUbyS}, we have 
\[U(K')\geq U(D')\geq (0, m), \text{~which implies~}\]  
\begin{equation}\label{eq1 n,m}
 U(D)=(n_D,m_D)\geq (n,m).
\end{equation} 

\noindent To obtain the upper bound, virtualize  crossings $c_1,c_2,\ldots,c_n$, and apply crossing change operations at crossings $d_2,d_4,\ldots, d_{2m}$ in $D'$. Then the resulting diagram is a diagram of the trivial knot, see Fig.~\ref{tmnu}, and hence
\begin{equation}\label{eq2 n,m}
 U(D)\leq (n,m).
 \end{equation}
 Equation~(\ref{eq1 n,m}) and equation~(\ref{eq2 n,m}) implies that $U(D)=(n,m)$.\\
\noindent Observe that,  all the crossings of $D$ are of same sign and having non-zero index. By Lemma~\ref{Mlemma}, $D$ is a diagram of $K$ having minimum number of classical crossings. \\
Consider odd positive integers $p_1\neq p_2$, satisfying 
\[p_1>q> n+2 \quad \text{and} \quad p_2>q> n+2. \]
 Let $D_i,~i\in\{1,2\}$ be the virtual knot diagram obtained from $D^{n,p_i,m},~i\in\{1,2\}$ by virtualizing $q$ number of crossing labeled with odd integers in the horizontal block. Clearly, the unknotting index of virtual knot diagram $D_i,~i\in\{1,2\}$ is $(n,m)$. Let $K_i$ be the virtual knot corresponding to diagram $D_i$, $i\in \{1,2\}$. Then by Lemma~\ref{Mlemma}, it is evident that  $D_i$ are a diagram of $K_i$ with minimum number of classical crossings. The total  number of classical crossings in $D_i$ are $(n+2m+p_i-q)$.  For $p_1\neq p_2$, $n+2m+p_1-q \neq n+2m+p_2-q$. Therefore, the diagrams $D_1$ and $D_2$ are not equivalent. Thus, the virtual knots $K_1$ and $K_2$ corresponding to $D_1$ and $D_2$, respectively, are not equivalent.  \\ 
For fixed $n,m \in \mathbb{N}$ and any $q> n+2$, there exists infinitely many odd integers $p$ satisfying $p>q> n+2$.
Hence, there exists infinitely many distinct virtual diagrams with unknotting index $(n,m)$. Therefore, the result holds.
\end{proof}
\begin{figure}[!ht] {\centering
\includegraphics[scale=0.5]{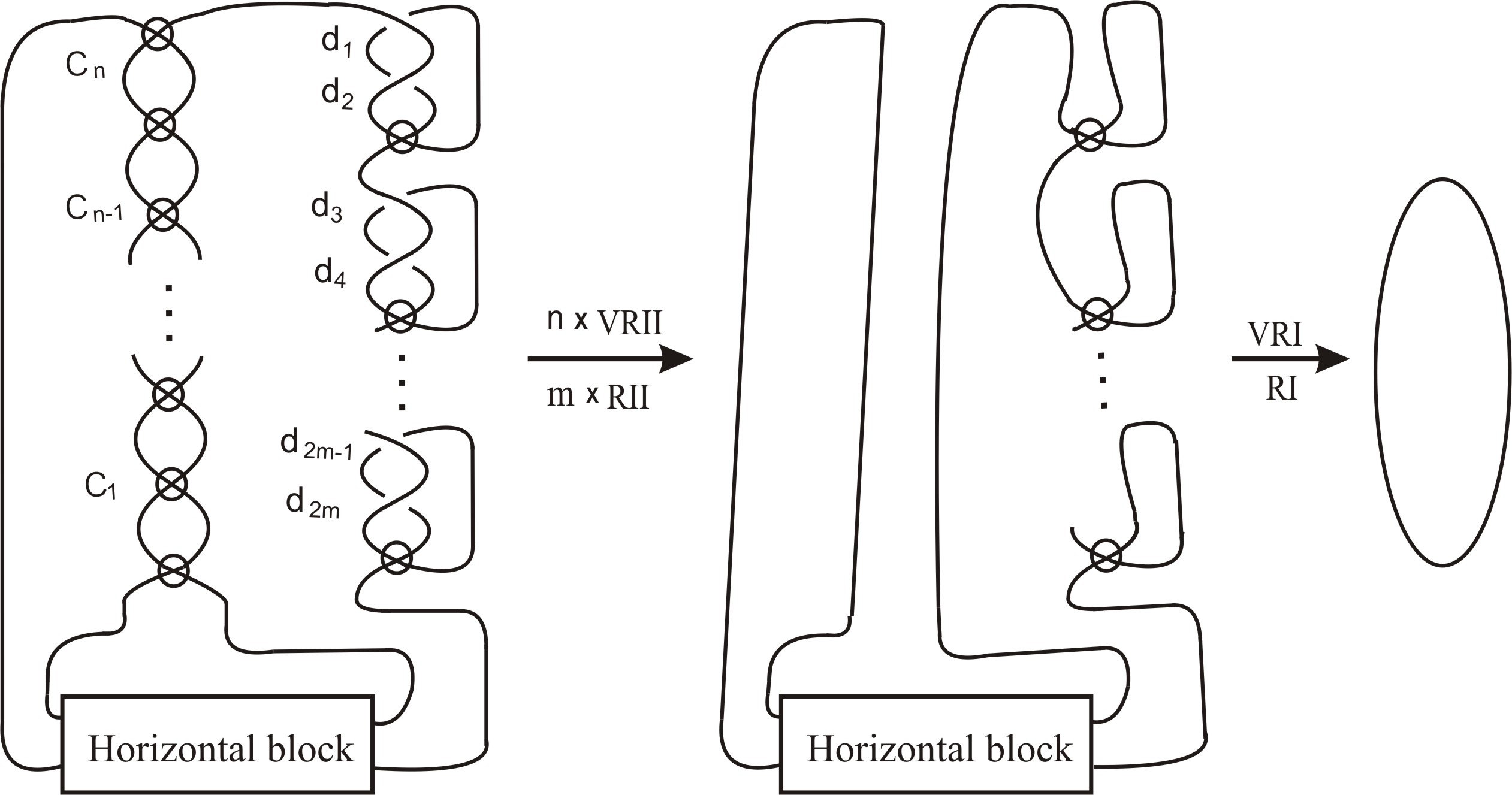}
 \vspace{.4cm}
\caption{A diagram of the trivial knot.} \label{tmnu}}\end{figure}
 \begin{figure}[!ht] 
{\centering
\includegraphics[scale=0.55]{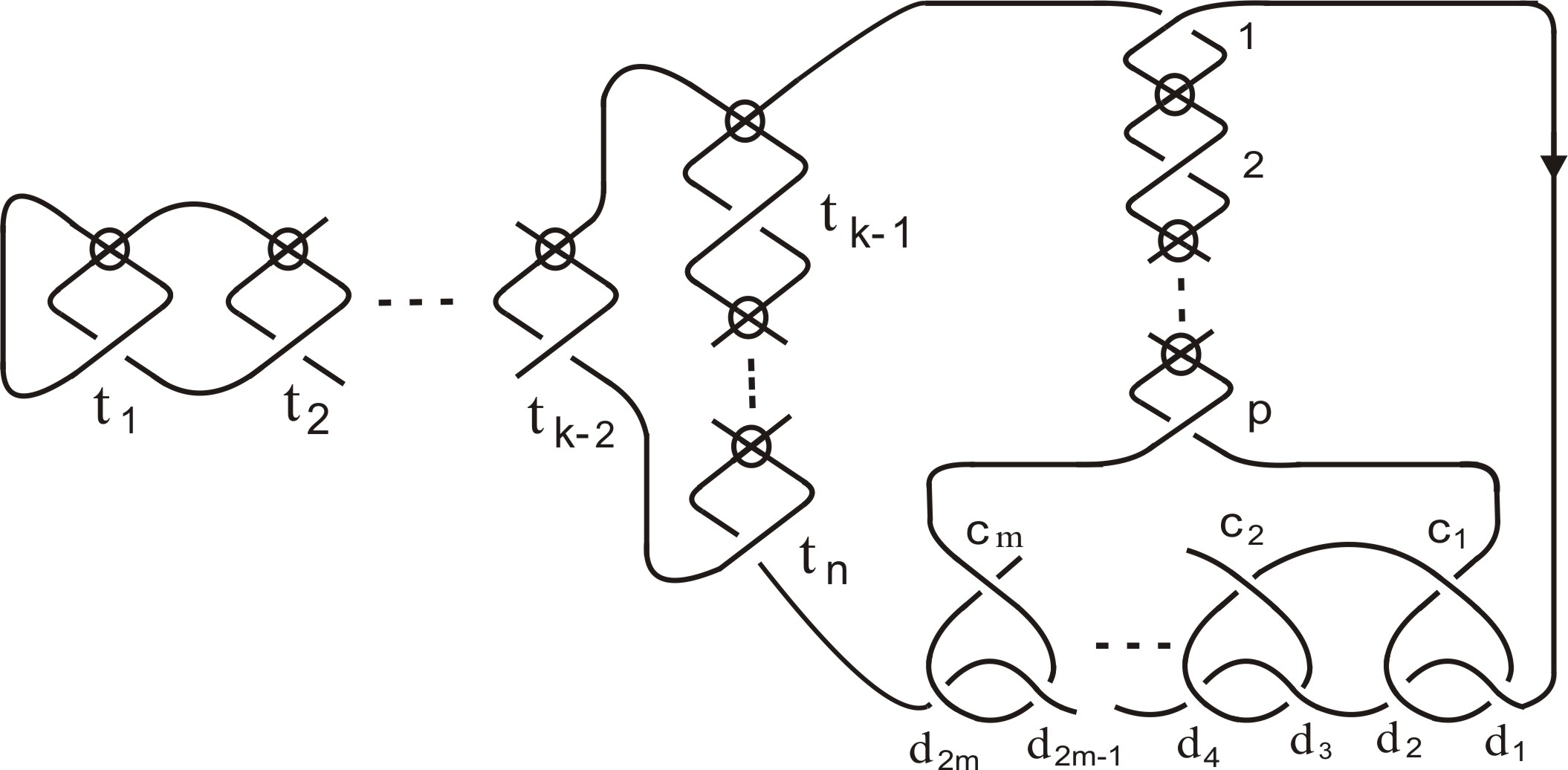} 
 \caption{Virtual link diagram $D=D_1\cup D_2\cup \cdots\cup D_k$. }
\label{figrf10}}
\end{figure}

\begin{remark}\label{rmk:(n,m)knot}
Let $D$ be a virtual knot diagram obtained from the diagram $D^{0,p,m}$ given in Fig.~\ref{mnu}, by virtualizing $q$ crossings with odd labelings in the horizontal block with $p>q>2$. It is evident that the diagram $D$ is equivalent to the diagram as shown in Fig.~\ref{mnu2} and $J_{-1}(D)=J_{1}(D)=-m$, $J_{k}(D)=0,~k\neq \pm 1$. If $K$ is the virtual knot represented by diagram $D$, then using Proposition~\ref{prop:boundsUbyS}(b), we have
 \[U(K)\geq \big(0, \dfrac{1}{2}\displaystyle \sum_{k \neq 0} |J_k(D)|\big)=\big(0,\dfrac{1}{2}(|J_{-1}(D)|+|J_{1}(D)|)\big)=(0,m).\]
 Observe that by applying crossing change operation at crossings $d_2, d_4,\ldots,d_{2m}$ in $D$, the resulting diagram becomes a diagram of the trivial knot. Thus $U(K)\leq U(D)\leq (0,m)$ and hence $U(K)=(0,m).$ 
 \end{remark}
\begin{theorem}\label{thm:(n,m)link}
 For any pair $(n,m)$ of positive integers, there exist infinitely many $k$-component virtual links with unknotting index $(n,m)$, where  $k$ is any integer $>1$.
\end{theorem}
\begin{proof}  Let $n,m, k, p \in \mathbb{N}$ with $p \geq 2$ and $k>1$. We prove this result by considering the two cases $k\leq  n$ and $k> n$.\\
{\underline{\it Case 1.}} Suppose that $k \leq n$. Consider a virtual link $L$ presented by diagram $D=D_1\cup D_2\cup \ldots\cup D_k$,  where $D_k=K^{m,p}$, as shown in Fig.~\ref{figrf10} and Fig.~\ref{figrf1}, respectively.  Let $t_1,t_2,\ldots,t_n$ be the linking crossings of $D$ as shown in Fig.~\ref{figrf10}. It is easy to see that there is no linking crossing between $D_i$ and $D_j$ when $j\neq i+1$. For $1\leq i \leq k-2 $, there is only one linking crossing, crossing $t_i$, between $D_i$ and $D_{i+1}$. Note that $t_{k-1},t_{k},\ldots, t_{n}$ are linking crossings between $D_{k-1}$ and $D_k$. Thus the span value of the diagram $D$ is given as
\begin{align*}
\operatorname{span}(D)&=\displaystyle \sum_{i\neq j} \operatorname{span}(D_i \cup D_{j})=\displaystyle \sum_{i=1}^{k-1} \operatorname{span}(D_i \cup D_{i+1})\\
&=\displaystyle \sum_{i=1}^{k-2} \operatorname{span}(D_i \cup D_{i+1})+\operatorname{span}(D_{k-1} \cup D_{k})\\
&=\left(\displaystyle \sum_{i=1}^{k-2} 1 \right)+(n-k+2)=n.
\end{align*}

\noindent By using Theorem~\ref{thm:bound_links}, minimum $n$ number of virtualizations are required to deform $D$ into the trivial link. These $n$ virtualizations must be at linking crossings. Since there are exactly $n$ linking crossings in the diagram $D$, it is evident that 
$\displaystyle \sum_{i \neq j} \ell_{D_{i} \cup D_{j}}=0$.
By using Theorem~\ref{thm:bound_links} and  equation~\ref{Eq0(0,m)}, we have 
\begin{align*} \hspace{2cm} U(L) &\geq \Big(\operatorname{span}(D), \displaystyle \sum_{i \neq j} \ell_{D_{i} \cup D_{j}} + \frac{1}{2} \sum_{i=1}^{k} \sum_{N \in \mathbb{Z}\setminus \{0\}} |  J_N(D_{i}) | \Big)\\
(18)\hspace{2.6cm}& =\left(n,0+ \frac{1}{2} \sum_{N \in \mathbb{Z}\setminus \{0\}} \abs{ J_N(K^{m,p})}\right)=(n,m). \hspace{2.6cm}
\end{align*}

\noindent For upper bound, by virtualizing all linking crossings in $D$ followed by crossing change operations at $d_2,d_4,\ldots,d_{2m}$ crossings, the resulting diagram becomes a diagram of the trivial link. Hence 
\[(19)\hspace{4cm} U(L)\leq U(D)\leq (n,m).\hspace{4cm} \]
Equation~(18) and equation~(19) implies that $U(L)=(n,m).$ \\
As discussed in the proof of Theorem~\ref{thm:(0,m)} that  the virtual knot diagrams $K^{m,p}$ for different values of $p\geq 2$, represent non-equivalent virtual knots. Thus different values of $p\geq 2$, yields distinct $k$-component virtual links whose unknotting index is $(n,m).$\\
 Moreover, there exist infinitely many positive integers $p\geq 2$. Thus, there exist infinitely many $k$-component virtual links with unknotting index $(n,m)$.\\

\noindent {\underline{\it Case 2.}} Suppose that $k > n$. Take any $k'$ such that $k'\leq n$. Then the result holds for $k'$ as discussed in {\it Case 1} of Theorem~\ref{thm:(n,m)link}. Let $L'$ be the $k'$- component link with unknotting index $(n,m)$ as shown in Fig.~\ref{figrf10}. Consider a $k$-component virtual link, say $L$, which is a disjoint union of $L'$ and $(k-k')$ trivial knots. Thus, $L$ is our desired $k$-component virtual link.  We can construct infinitely many $k$-component virtual links with unknotting index $(n,m)$ corresponding to infinitely many $k'$-component virtual links with unknotting index $(n,m)$. Hence the proof of the result follows. 
\end{proof}

\begin{figure}[!ht] 
{\centering
\includegraphics[scale=0.55]{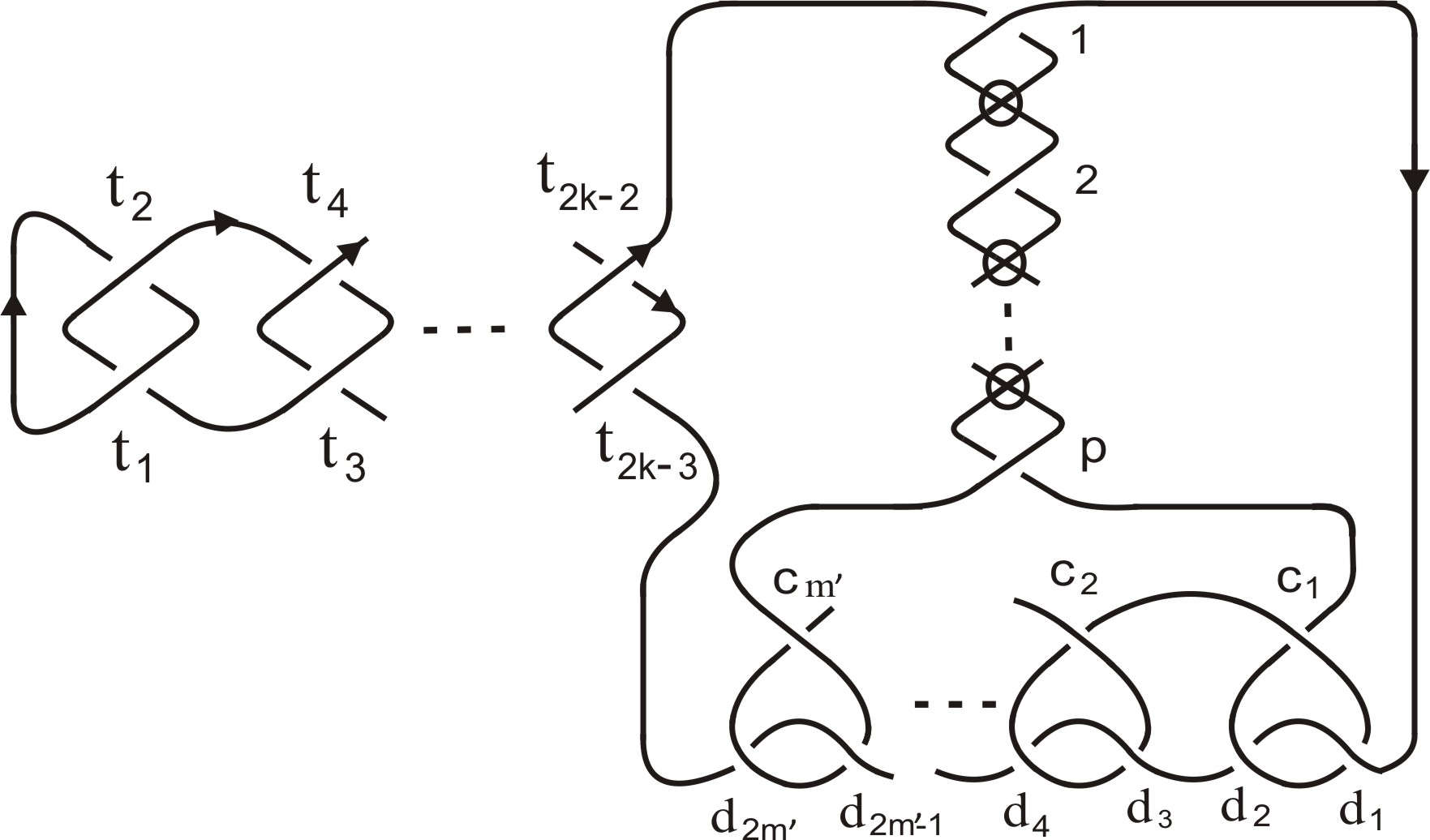}
\caption{ Virtual link diagram $D=D_1\cup D_2\cup \cdots \cup D_k$. }\label{(0,m)}
 } \end{figure}
 
\begin{proposition}\label{prop1:(n,m)link}
  Let $n,m,k\in \mathbb{N}$ with $k>1$. Then there exist infinite families of $k$-component virtual links with unknotting index $(n,0)$ and $(0,m)$, respectively.
\end{proposition}
\begin{proof}
 Let $n,m, k, p \in \mathbb{N}$ with $k>1$ and odd $p \geq 2$. First we establish the result for the pair $(0,m)$, by considering case $k\leq  n$.\\
{\underline{\it Case 1.}} Suppose that $k \leq n$. Consider a virtual link $L$ presented by diagram $D=D_1\cup D_2\cup \ldots \cup D_k$, as shown in Fig.~\ref{(0,m)}, and the diagram $D_k=K^{m',p}$, as shown in Fig.~\ref{figrf1}, with $m'=(m-k+1)$. Let $t_1,t_2,\ldots,t_{2k-2}$ be the linking crossings of $D$ as shown in Fig.~\ref{(0,m)}. It is easy to see that there is no linking crossing between $D_i$ and $D_j$ when $j\neq i+1$. For $1\leq i \leq k-1 $, there are two linking crossings, crossing $t_i$ and $t_{2i}$, between $D_i$ and $D_{i+1}$ and 
\[\operatorname{sgn}(t_i)=\operatorname{sgn}(t_{2i})=-1.\] 
It is evident that $\operatorname{span}(D_i \cup D_{i+1})=0$ for $1\leq i \leq k-1 $, and 
\begin{align*}
\operatorname{span}(D)&=\displaystyle \sum_{i\neq j} \operatorname{span}(D_i \cup D_{j})=\displaystyle \sum_{i=1}^{k-1} \operatorname{span}(D_i \cup D_{i+1})=\displaystyle \sum_{i=1}^{k-1} 0=0 .
\end{align*}

\noindent  Since  $\operatorname{span}(D)=0$, by definition $\ell_{D_{i} \cup D_{j}}=\operatorname{lk}(D_i \cup D_{j}),$ and 
\begin{align*}\displaystyle \sum_{i \neq j} \ell_{D_{i} \cup D_{j}}&=\displaystyle \sum_{i \neq j}|\operatorname{lk}(D_i \cup D_j)|=\displaystyle \sum_{i=1}^{k-1} |\operatorname{lk}(D_i \cup D_{i+1})|\\
&=\displaystyle \sum_{i=1}^{k-1}\left|\dfrac{1}{2}(\operatorname{sgn}(t_{2i-1})+\operatorname{sgn}(t_{2i}))\right|=\displaystyle \sum_{i=1}^{k-1}{|-1|}=(k-1).
\end{align*}
Since $D_i,~i\neq k$ is a diagram of the trivial knot and $D_k=K^{m',p}$, 
\[\displaystyle\sum_{N \in \mathbb{Z}\setminus \{0\}} \abs{ J_N(D_i)}=0 \text{~for} ~i\neq k, \text{~and using equation~(\ref{Eq0(0,m)}), we have}\]
\[ \sum_{i=1}^{k} \sum_{N \in \mathbb{Z}\setminus \{0\}} |  J_N(D_{i}) |=\sum_{N \in \mathbb{Z}\setminus \{0\}} \abs{ J_N(D_k)}=\sum_{N \in \mathbb{Z}\setminus \{0\}} \abs{ J_N(K^{m',p})}=2m'=2(m-k+1).\]
By using Theorem~\ref{thm:bound_links}, we have 
\begin{align*} \hspace{2cm} U(L) &\geq \Big(\operatorname{span}(D), \displaystyle \sum_{i \neq j} \ell_{D_{i} \cup D_{j}} + \frac{1}{2} \sum_{i=1}^{k} \sum_{N \in \mathbb{Z}\setminus \{0\}} |  J_N(D_{i}) | \Big)\\
(20)\hspace{2.6cm}& =\left(0,(k-1)+ \frac{1}{2} 2(m-k+1) \right)=(0,m). \hspace{2.6cm}
\end{align*}

\noindent For upper bound, by applying $m$ crossing change operations at $t_2, t_4,\ldots,t_{2k-2}$ and $d_2,d_4,\ldots,d_{2m'}$ crossings in the diagram $D$ as shown in Fig.~\ref{(0,m)}, the resulting diagram becomes a diagram of the trivial link. Hence, 
\[(21)\hspace{4cm} U(L)\leq U(D)\leq (0,m).\hspace{4cm} \]
Equation~(20) and equation~(21) implies that $U(L)=(0,m).$ \\
As discussed in the proof of Theorem~\ref{thm:(0,m)},  the virtual knot diagrams $K^{m',p}$, for different values of $p\geq 2$, represent non-equivalent virtual knots. Thus different values of $p\geq 2$, yields distinct $k$-component virtual links whose unknotting index is $(0,m).$\\
 Moreover, there exist infinitely many odd integers $p\geq 2$. Thus, there exist infinitely many $k$-component virtual links with unknotting index $(0,m)$.\\

\begin{figure}[!ht] 
{\centering
\includegraphics[scale=0.55]{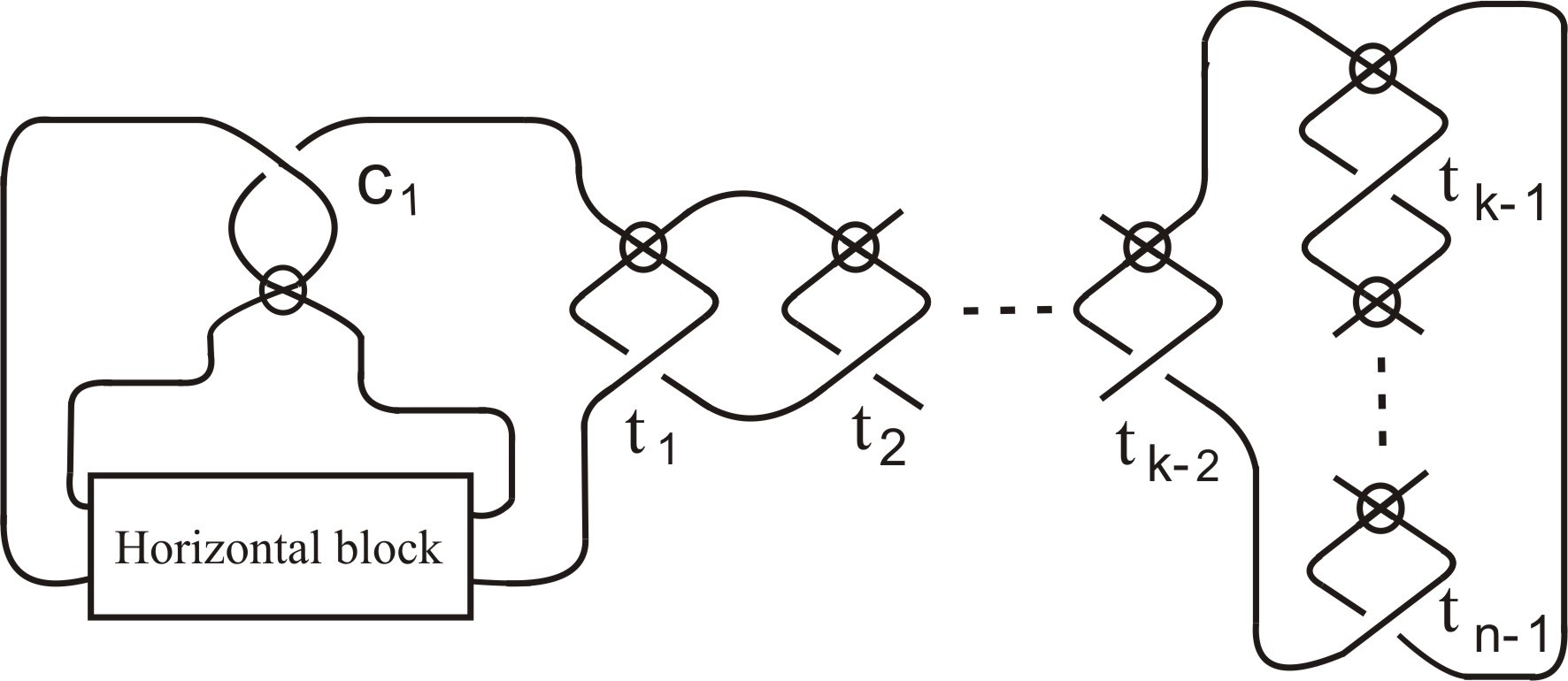}
\caption{ Virtual link diagram $D=D_1\cup D_2\cup \cdots \cup D_k$, where horizontal block contains crossings same as the crossings in the horizontal block of $D^{1,p}_{q,0}$. }\label{(n,0)}
 } \end{figure}
\noindent {\underline{\it Case 2.}} In this case, we establish the result for the pair $(n,0)$, by considering case $k\leq  n$. Consider a virtual link $L$ presented by the diagram  $D=D_1\cup D_2\cup\cdots\cup D_k$ as shown in Fig.~\ref{(n,0)},  with $D_1=D^{1,p}_{q,0}$ for any $q$ satisfying $p>q>3$.  Let $t_1,t_2,\ldots,t_{n-1}$ be the linking crossings of $D$ as shown in Fig.~\ref{(n,0)}. It is easy to see that there is no linking crossing between $D_i$ and $D_j$ when $j\neq i+1$. For $1\leq i \leq k-2 $, there is only one linking crossing, crossing $t_i$, between $D_i$ and $D_{i+1}$. Note that $t_{k-1},t_{k},\ldots, t_{n-1}$ are linking crossings between $D_{k-1}$ and $D_k$. Thus it is easy to compute the span value of the diagram $D$ as
\begin{align*}
\operatorname{span}(D)&=\displaystyle \sum_{i\neq j} \operatorname{span}(D_i \cup D_{j})=\displaystyle \sum_{i=1}^{k-1} \operatorname{span}(D_i \cup D_{i+1})\\
&=\displaystyle \sum_{i=1}^{k-2} \operatorname{span}(D_i \cup D_{i+1})+\operatorname{span}(D_{k-1} \cup D_{k})\\
&=\left(\displaystyle \sum_{i=1}^{k-2} 1 \right)+(n-k+1)=n-1.
\end{align*}

\noindent By using Remark~\ref{rem:span_inv}, minimum $(n-1)$ number of linking crossings are required to virtualize in $D$ to deform $D$ into a split link. Additionally, in the diagram $D_1=D^{1,p}_{q,0}$, $q>2$. Then by Lemma~\ref{lem-(n,m)} and Remark~\ref{rmk:(n,m)knot}, $U(K)\geq (1,0)$, where $K$ is the virtual knot corresponding to $D^{1,p}_{q,0}.$ That means minimum one crossing need to be  virtualized in  $D_1=D^{1,p}_{q,0}$ to deform $D_k=D^{1,p}_{q,0}$ into the trivial knot. Thus using Remark~\ref{rem:span_inv}, minimum $\operatorname{span}(D)+1=n$ number of crossings are required to be virtualized in $D$ to deform  $D$ into a trivial link. Since span value is a virtual link invariant, we have  
\[(22)\hspace{5cm} U(L)>(n,0).\hspace{5cm}\]
\noindent For upper bound, by virtualizing all $(n-1)$ linking crossings in $D$ and crossing $c_1$, the resulting diagram becomes a diagram of the trivial link. Hence 
\[(23)\hspace{4cm} U(L)\leq U(D)\leq (n,0).\hspace{4cm} \]
Equation~(22)  and equation~(23) implies that $U(L)=(n,0).$ \\
As discussed in the proof of Theorem~\ref{thm:(1,0)knot},   the virtual knot diagrams $D_1=D^{1,p}_{q,0}$ for a fixed $q>3$, and different values of odd $p>q$, represent non-equivalent virtual knots. Thus different values of odd $p$ satisfying $p>q>3$, yields distinct $k$-component virtual links whose unknotting index is $(n,0).$\\
 Moreover, there exist infinitely many odd positive integers satisfying $p>q>3$, for a fixed $q$. Thus, there exist infinitely many $k$-component virtual links with unknotting index $(n,0)$.\\

\noindent {\underline{\it Case 3.}} In this case, we establish the result for the pairs $(0,m)$ and $(n,0)$, by considering case $k> n$.  Take any $k'$ such that $k'\leq n$. Then the result holds for $k'$ as discussed in {\it Case 1}, and  {\it Case 2} of Theorem~\ref{prop1:(n,m)link}. Let $L'_{1}$ and $L'_{2}$ be the $k'$- component virtual links with unknotting index $(0,m)$ and $(n,0)$ as shown in Fig.~\ref{(0,m)}, and Fig.~\ref{(n,0)} with $q>3$, respectively. Consider a $k$-component virtual link, say $L_i,~i\in\{1,2\}$, which is a disjoint union of $L'_{i}$ and $(k-k')$ trivial knots. Thus $L_{1}$ and $L_2$ are  our desired $k$-component virtual links. We can construct infinitely many $k$-component virtual links with unknotting index $(0,m)$ and $(n,0)$, respectively. Hence the proof of the result follows. 
\end{proof}
\section{Conclusion}\label{conclusion}

\noindent In this paper, we address the challenge posed by K. Kaur et al. in \cite{kaur2019unknottingknot} concerning the existence of virtual knots with unknotting index $(n,m)$. We establish  infinite families of $k(>1)$-component virtual links with unknotting index $(n,m)$, for any given pair of non-negative integers $(n,m)$. However, for $k=1$, proving the existence of a virtual knot with unknotting index $(n,m)$ remains difficult.
Specifically, we identify infinite families of virtual knots with unknotting indices $(0,m)$ and $(1,0)$, respectively. For any pair $(n,m)$ of positive integers, we construct infinitely many distinct virtual knot diagrams with unknotting index $(n,m)$.

\noindent Determining the unknotting index of the virtual knots corresponding to the diagrams constructed in Theorem~\ref{thm:(n,m)knots} presents a challenging problem.    If we consider $n=1$ and $m=0$, then the diagram constructed in Theorem~\ref{thm:(n,m)knots} coincides with the diagram $D^{1,p}_{q,0}$ constructed in Theorem~\ref{thm:(1,0)knot}. By Theorem~\ref{thm:(1,0)knot}, $U(K)=U(D^{1,p}_{q,0})=U(D)=(1,0)$, where $K$ is the virtual knot corresponding to the diagram $D$. If we consider $n=0$ and $m\in \mathbb{N}$, then the diagram $D$ constructed in Theorem~\ref{thm:(n,m)knots} is equivalent to the diagram  shown in Fig.~\ref{mnu2}. By Remark~\ref{rmk:(n,m)knot}, $U(K)=U(D)=(0,m)$, where $K$ is the virtual knot corresponding to the diagram $D$.\\
\noindent A natural question arises: For a virtual knot $K$ represented by the diagram $D$ constructed in Theorem 3.7, does  $U(K)=U(D)$ hold for any pair of positive integers $(n,m)?$\\

\noindent \section{Acknowledgment}

\noindent The authors would like to thank Prof. Jitender Singh for his remarks and suggestions during  the preparation of this paper. First author was supported by  National Board of Higher Mathematics (NBHM)(Ref. No. 0204/16(11)/2022/R $\&$D-II/11983), Government of India. Second author acknowledges the support given by the Science and Engineering Board (SERB), Department of Science and Technology, Government of India Under the Mathematical Research Impact Centric Support (MATRICS) grant-in-aid with F.No.MTR/2021/00394 and by the NBHM, Government of India under grant-in-aid with F.No.02011/2/ 20223 NBHM(R.P.)/R$\&$D II/970.

\end{document}